\documentclass[12pt, reqno]{jams-l}

\usepackage{amsmath,enumerate,amsfonts,amssymb,color,graphicx,amsthm}

\usepackage[colorlinks=true, allcolors=black]{hyperref}
\usepackage{todonotes}
\usepackage[normalem]{ulem}
\numberwithin{equation}{section}
\usepackage{cite}
\usepackage{mathtools}
\usepackage[toc,page]{appendix}
\usepackage[margin=3cm]{geometry}

\newlength{\leftstackrelawd}
\newlength{\leftstackrelbwd}
\def\leftstackrel#1#2{\settowidth{\leftstackrelawd}%
	{${{}^{#1}}$}\settowidth{\leftstackrelbwd}{$#2$}%
	\addtolength{\leftstackrelawd}{-\leftstackrelbwd}%
	\leavevmode\ifthenelse{\lengthtest{\leftstackrelawd>0pt}}%
	{\kern-.5\leftstackrelawd}{}\mathrel{\mathop{#2}\limits^{#1}}}

\theoremstyle{plain}
\newtheorem{thm}{Theorem}[section]
\newtheorem{lem}[thm]{Lemma}
\newtheorem{cor}[thm]{Corollary}
\newtheorem{prop}[thm]{Proposition}
\newtheorem*{thm*}{Theorem}

\theoremstyle{definition}
\newtheorem{defn}[thm]{Definition}

\newtheorem{rmk}[thm]{Remark}
\newtheorem{?}[thm]{Problem}

\newtheorem{innercustomthm}{Theorem}
\newenvironment{customthm}[1]
{\renewcommand\theinnercustomthm{#1}\innercustomthm}
{\endinnercustomthm}

\newtheorem{innercustomlem}{Lemma}
\newenvironment{customlem}[1]
{\renewcommand\theinnercustomlem{#1}\innercustomlem}
{\endinnercustomlem}

\newcommand{\ep}{\varepsilon}
\renewcommand{\phi}{\varphi}
\renewcommand{\epsilon}{\varepsilon}
\makeatletter
\def\@cite#1#2{[\textbf{#1\if@tempswa , #2\fi}]}
\def\@biblabel#1{[\textbf{#1}]}
\makeatother

\makeatletter
\newcommand*{\defeq}{\mathrel{\rlap{%
			\raisebox{0.3ex}{$\m@th\cdot$}}%
		\raisebox{-0.3ex}{$\m@th\cdot$}}%
	=}
\makeatother

\makeatletter
\newcommand*{\eqdef}{=\mathrel{\rlap{%
			\raisebox{0.3ex}{$\m@th\cdot$}}%
		\raisebox{-0.3ex}{$\m@th\cdot$}}%
	}
\makeatother

\newcounter{marnote}

\makeatletter
\def\underbracex#1#2{\mathop{\vtop{\m@th\ialign{##\crcr
				$\hfil\displaystyle{#2}\hfil$\crcr
				\noalign{\kern3\p@\nointerlineskip}%
				#1\crcr\noalign{\kern3\p@}}}}\limits}

\def\upbracefilla{$\m@th \setbox\z@\hbox{$\braceld$}%
	\bracelu\leaders\vrule \@height\ht\z@ \@depth\z@\hfill 
	\kern\p@\vrule \@width\p@\kern\p@\vrule \@width\p@\kern\p@\vrule \@width\p@
	$}

\def\upbracefillb{$\m@th \setbox\z@\hbox{$\braceld$}%
	\vrule \@width\p@\kern\p@\vrule \@width\p@\kern\p@\vrule \@width\p@\kern\p@
	\leaders\vrule \@height\ht\z@ \@depth\z@\hfill\bracerd
	\braceld\leaders\vrule \@height\ht\z@ \@depth\z@\hfill
	\kern\p@\vrule \@width\p@\kern\p@\vrule \@width\p@\kern\p@\vrule \@width\p@
	$}

\def\upbracefillc{$\m@th \setbox\z@\hbox{$\braceld$}%
	\vrule \@width\p@\kern\p@\vrule \@width\p@\kern\p@\vrule \@width\p@\kern\p@
	\leaders\vrule \@height\ht\z@ \@depth\z@\hfill
	\kern\p@\vrule \@width\p@\kern\p@\vrule \@width\p@\kern\p@\vrule \@width\p@
	$}

\def\upbracefilld{$\m@th \setbox\z@\hbox{$\braceld$}%
	\vrule \@width\p@\kern\p@\vrule \@width\p@\kern\p@\vrule \@width\p@\kern\p@
	\leaders\vrule \@height\ht\z@ \@depth\z@\hfill\braceru$}

\def\upbracefillbd{$\m@th \setbox\z@\hbox{$\braceld$}%
	\vrule \@width\p@\kern\p@\vrule \@width\p@\kern\p@\vrule \@width\p@\kern\p@
	\bracerd\braceld
	\leaders\vrule \@height\ht\z@ \@depth\z@\hfill\braceru$}

\makeatother

\makeatletter
\def\l@subsection{\@tocline{2}{0pt}{2.5pc}{5pc}{}}
\makeatother

\begin{document}

	\title[Differential inclusions for the Schouten tensor]{Differential inclusions for the Schouten tensor and nonlinear eigenvalue problems \linebreak in conformal geometry} 
	
\author{Jonah A. J. Duncan}
\address{Johns Hopkins University, 404 Krieger Hall, Department of Mathematics, 3400 N. Charles Street, Baltimore, MD 21218, US.}
\curraddr{}
\email{jdunca33@jhu.edu}
\thanks{}
	
	\author{Luc Nguyen}
	\address{Mathematical Institute and St Edmund Hall, University of Oxford, Andrew Wiles Building, Radcliffe Observatory Quarter, Woodstock Road, OX2 6GG, UK.}
	\curraddr{}
	\email{luc.nguyen@maths.ox.ac.uk}
	\thanks{}
	
	\date{}

	\maketitle

	\vspace*{-8mm}\begin{abstract}
		Let $g_0$ be a smooth Riemannian metric on a closed manifold $M^n$ of dimension $n\geq 3$. We study the existence of a smooth metric $g$ conformal to $g_0$ whose Schouten tensor $A_g$ satisfies the differential inclusion $\lambda(g^{-1}A_g)\in\Gamma$ on $M^n$, where $\Gamma\subset\mathbb{R}^n$ is a cone satisfying standard assumptions. Inclusions of this type are often assumed in the existence theory for fully nonlinear elliptic equations in conformal geometry. We assume the existence of a continuous metric $g_1$  conformal to $g_0$ satisfying $\lambda(g_1^{-1}A_{g_1})\in\overline{\Gamma'}$ in the viscosity sense on $M^n$, together with a nondegenerate ellipticity condition, where $\Gamma' = \Gamma$ or $\Gamma'$ is a cone slightly smaller than $\Gamma$. In fact, we prove not only the existence of metrics satisfying such differential inclusions, but also existence and uniqueness results for fully nonlinear eigenvalue problems for the Schouten tensor. We also give a number of geometric applications of our results. We show that the solvability of the $\sigma_2$-Yamabe problem is equivalent to positivity of a nonlinear eigenvalue for the $\sigma_2$-operator in three dimensions. We also give a generalisation of a theorem of Aubin \& Ehrlick on pinching of the Ricci curvature, and an application in the study of Green's functions for fully nonlinear Yamabe problems.

	\end{abstract}
	
	\setcounter{tocdepth}{2}
	\vspace*{-2mm}\tableofcontents

	\section{Introduction}
	
	\subsection{Main results}
	Let $M^n$ be a closed manifold of dimension $n\geq 3$, $g_0$ a smooth Riemannian metric on $M^n$ and $[g_0]$ the set of smooth metrics conformal to $g_0$. The Yamabe problem, solved through the combined works of Yamabe \cite{Yam60}, Trudinger \cite{Tru68}, Aubin \cite{Aub76} and Schoen \cite{Sch84}, asserts that $[g_0]$ contains a metric $g$ of constant scalar curvature $R_g$, where the constant is positive (resp.~negative, resp.~zero) if and only if $[g_0]$ contains a metric of positive (resp.~negative, resp.~zero) scalar curvature. It is well-known that the sign of any such constant coincides with the sign of the first eigenvalue of the conformal Laplacian of any metric in $[g_0]$, and the sign of the Yamabe invariant
	\begin{equation*}
	Y(M^n,[g_0]) = \inf_{g\in[g_0]} \frac{\int_{M^n}R_g\,dv_g}{\operatorname{Vol}(M^n,g)^{\frac{n-2}{n}}}.
	\end{equation*}

	Since the work of Viaclovsky \cite{Via00a} and Chang, Gursky \& Yang \cite{CGY02a}, there has been significant interest in fully nonlinear generalisations of the Yamabe problem involving the eigenvalues of the trace-modified Schouten tensor $A_{g_0}^t$, defined for $t\leq 1$ by 
	\begin{equation}\label{300}
	A_{g_0}^t =  \frac{1}{n-2}\bigg(\operatorname{Ric}_{g_0} - \frac{tR_{g_0}}{2(n-1)}g_0\bigg). 
	\end{equation}
	Here, $\operatorname{Ric}_{g_0}$ denotes the Ricci curvature tensor of $g_0$. When $t=1$, $A_{g_0}^1$ (henceforth denoted $A_{g_0}$) is the Schouten tensor, which arises in the Ricci decomposition of the Riemann curvature tensor. Of particular interest are elliptic equations of the form
	\begin{equation}\label{16}
	\sigma_k(\lambda(g_u^{-1}A_{g_u}^t)) = \psi(x,u), \quad \lambda(g_u^{-1}A_{g_u}^t)\in\Gamma_k^+ \quad \text{on }M^n,
	\end{equation}
	where $t\leq 1$ is fixed, $\psi>0$ is a prescribed function and $g_u=e^{-2u}g_0\in[g_0]$ is the unknown metric. In \eqref{16}, $\lambda(g_u^{-1}A_{g_u}^t)$ denotes the eigenvalues of the $(1,1)$-tensor $g_u^{-1}A_{g_u}^t$, $\sigma_k:\mathbb{R}^n\rightarrow \mathbb{R}$ is the $k$'th elementary symmetric polynomial, and 
	\begin{equation*}
	\Gamma_k^+ = \{\lambda = (\lambda_1,\dots,\lambda_n)\in\mathbb{R}^n:\sigma_j(\lambda)>0\text{ for all }1\leq j \leq k\}. 
	\end{equation*}
	Note that $A_{g_u}^t$ and $A_{g_0}^t$ are related by the conformal transformation law
	\begin{equation}\label{301}
	A_{g_u}^t = \nabla_{g_0}^2 u + \frac{1-t}{n-2}\Delta_{g_0} u\,g_0 - \frac{2-t}{2}|du|_{g_0}^2g_0 + du\otimes du + A_{g_0}^t. 
	\end{equation}
	
	When $k=1$ and $\psi$ is a positive constant, \eqref{16} is the Yamabe equation in the case of positive Yamabe invariant. When $2\leq k \leq n$ and $t\leq1$, \eqref{16} is fully nonlinear and non-uniformly elliptic, and when $\psi$ is constant it is often referred to as the (trace-modified, when $t<1$) $\sigma_k$-Yamabe equation\footnote{Equation \eqref{16} is also elliptic when $t\geq n-1$, although this case is different in nature due to negativity of the scalar curvature, and is not considered in this paper.}.  
	
	In existing literature addressing \eqref{16}, it has been customary to assume that there exists a conformal metric $g\in[g_0]$ satisfying
	\begin{equation}\label{94}
	\lambda(g^{-1}A_{g}^t)\in\Gamma_k^+ \quad \text{on }M^n. 
	\end{equation} 
	Note that when $k=1$, \eqref{94} is precisely the assumption that $g$ has positive scalar curvature. Under the assumption \eqref{94} with $t=1$, the $\sigma_k$-Yamabe equation has been solved in various cases -- see e.g.~Chang, Gursky \& Yang \cite{CGY02b}, Ge \& Wang \cite{GW06}, Guan \& Wang \cite{GW03a}, Gursky \& Viaclovsky \cite{GV07}, Li \& Li \cite{LL03}, Li \& Nguyen \cite{LN14}, and Sheng, Trudinger \& Wang \cite{STW07}. More generally, under \eqref{94} the trace-modified $\sigma_k$-Yamabe equation has also been solved in various cases -- see e.g.~\cite{LL03, GV07, LN14}. For recent related works, see e.g.~\cite{GS18, FW19, CW19, LN20, LNW20, DN21, LNW21} and the references therein.

	It is an interesting and important problem to determine when $[g_0]$ contains a smooth metric satisfying \eqref{94} for $k\geq 2$; we will discuss previous integral-type results related to this problem slightly later in the introduction. In this paper, we are interested in using viscosity-type conditions to establish the existence of a smooth metric in $[g_0]$ satisfying \eqref{94}. Namely, we assume the existence of a continuous metric $g_{\hat{u}} = e^{-2\hat{u}}g_0$ satisfying \begin{equation}\label{40'}
	\lambda(g_{\hat{u}}^{-1}A_{g_{\hat{u}}})\in \overline{\Gamma_k^+} \quad\text{ in the viscosity sense on }M^n
	\end{equation}
	(see Section \ref{scp} for the meaning of \eqref{40'}). The notion of viscosity solutions to fully nonlinear Yamabe equations was first studied by Li in \cite{Li09}. We highlight that many of our results in this paper are new even in the case that $g_{\hat{u}}$ is smooth. Moreover, in light of \cite{Li09} and subsequent work on viscosity solutions to fully nonlinear Yamabe equations, such an additional smoothness assumption would not substantially simplify our treatment.
	
	Part of our motivation for the condition \eqref{40'} stems from the expectation that, in the study of compactness issues for \eqref{16}, an appropriately rescaled sequence of solutions to \eqref{16} (either for $t=1$ or as $t\rightarrow 1$) converges to a possibly non-smooth metric $g_{\hat{u}} = e^{-2\hat{u}}g_0$ satisfying \eqref{40'}. This has been observed in a number of situations, see for instance \cite{GV07, TW09, LL03, LN14, LN20, CGY02b, GLW10}. Since compactness issues play an important role in the existence theory for \eqref{16}, it is desirable to understand the gap between \eqref{40'} and \eqref{94}. The study of continuous metrics satisfying \eqref{40'} is also closely related to the study of Green's functions for fully nonlinear Yamabe problems -- see the recent work of Li \& Nguyen \cite{LN20}, and Theorem \ref{305} below for a related result.
	
	Our first main result is as follows: 
	
	\begin{thm}\label{f'}
		Let $(M^n,g_0)$ be a smooth, closed Riemannian manifold of dimension $n\geq 3$ and let $2\leq k \leq n$. Suppose there exists a metric $g_{\hat{u}} = e^{-2\hat{u}}g_0$, $\hat{u}\in C^0(M^n)$, satisfying \eqref{40'}. Then the following statements hold: \medskip
		\begin{enumerate}
			\item[1.] For given $t<1$, there exists a smooth metric $g_{t}\in[g_0]$ satisfying $\lambda(g_{t}^{-1}A_{g_{t}}^{t})\in\Gamma_k^+$ on $M^n$ if and only if $\hat{u}$ is not a smooth solution to $\operatorname{Ric}_{g_{\hat{u}}}\equiv 0$ on $M^n$. \medskip
			
			\item[2.] There exists a smooth metric $g\in[g_0]$ satisfying $\lambda(g^{-1}A_g)\in\Gamma_k^+$ on $M^n$ if and only if there is no $C^{1,1}$ metric $g$ conformal to $g_0$ satisfying $\lambda(g^{-1}A_g)\in\partial\Gamma_k^+$ a.e.~on $M^n$. 
		\end{enumerate} 
	\end{thm}

\begin{rmk}
	We note that, by the strong comparison principle, if there exists a smooth metric $g\in[g_0]$ satisfying $\lambda(g^{-1}A_g)\in\Gamma_k^+$ on $M^n$, then $Y(M^n,[g_0])>0$ and there is no $C^0$ metric $g$ conformal to $g_0$ satisfying $\lambda(g^{-1}A_g)\in\mathbb{R}^n\backslash\Gamma_k^+$ in the viscosity sense on $M^n$. The converse also holds (without assuming \eqref{40'}) -- see Theorem \ref{201}. 
\end{rmk}

To put Theorem \ref{f'} into context, we now briefly discuss some previous work on the existence of conformal metrics satisfying \eqref{94}. In \cite{CGY02a}, Chang, Gursky and Yang studied the case $k=2$, $n=4$ and $t=1$: they showed that if $Y(M^4,[g_0])>0$ and $\int_{M^4}\sigma_2(\lambda(g_0^{-1}A_{g_0}))\,dv_{g_0}>0$, then there exists a conformal metric $g\in[g_0]$ satisfying $\lambda(g^{-1}A_g)\in\Gamma_2^+$ on $M^4$. An alternative proof encompassing the case $t\leq 1$ was given by Gursky \& Viaclovsky in \cite{GV03}. Existence results under similar integral-type conditions were later established in three dimensions by Catino \& Djadli \cite{CD10} and Ge, Lin \& Wang \cite{GLW10}. In \cite{GLW10}, under the assumption that $R_{g_0}>0$ and $\int_{M^3}\sigma_2(\lambda(g_0^{-1}A_{g_0}))\,dv_{g_0}> 0$, the authors showed the existence of a smooth metric $g\in[g_0]$ satisfying $\lambda(g^{-1}A_g)\in\Gamma_2^+$ on $M^3$. If $g_0$ only satisfies $R_{g_0}>0$ and $\int_{M^3}\sigma_2(\lambda(g_0^{-1}A_{g_0}))\,dv_{g_0} \geq 0$, and $g_0$ cannot be conformally deformed to a different background metric for which the previous case applies, then $g_0$ is an optimiser for the functional $Y_{2,1}([g_0])$ considered in \cite{GLW10} and therefore satisfies $\lambda(g_0^{-1}A_{g_0})\in\partial\Gamma_2^+$ on $M^3$. It follows that $\lambda(g_0^{-1}A_{g_0}^t)\in\Gamma_2^+$ for all $t<1$; compare with \cite{CD10} for $t\leq t_0\approx 0.7$. In dimensions $n\geq 5$, similar existence results are only known to hold under integral-type conditions which are assumed to hold for \textit{all} metrics in $[g_0]$ -- see \cite{She08, GLW10}. Clearly, for continuous metrics $g_{\hat{u}}$, integral quantities such as $\int_{M^n}\sigma_k(\lambda(g_{\hat{u}}^{-1}A_{g_{\hat{u}}}))\,dv_{g_{\hat{u}}}$ may not even be well-defined\footnote{We will see later that there is a connection between our viscosity-type condition \eqref{40'} and certain integral conditions. In particular, we recover and extend the existence results of Ge, Lin \& Wang \cite{GLW10} and Sheng \cite{She08} for $k=2$ in dimensions $n\geq 5$ -- see Theorem \ref{4'} in Section \ref{52}.}.
	
	The following is an immediate consequence of the first statement in Theorem \ref{f'}:
	
	\begin{cor}\label{504'}
		In addition to the hypotheses of Theorem \ref{f'} suppose $Y(M^n,[g_0])>0$. Then for each $t<1$, there exists a smooth metric $g_{t}\in[g_0]$ with $\lambda(g_{t}^{-1}A_{g_{t}}^{t})\in\Gamma_k^+$ on $M^n$.
	\end{cor}

As far as the authors are aware, even in the case that $\hat{u}$ is smooth, Corollary \ref{504'} was not previously known in dimensions $n\geq 4$ (the case $n=3$ follows from \cite{GLW10}, in light of the discussion above).

	As remarked above, the strong comparison principle implies that if $g_{\hat{u}}$ satisfies \linebreak $\lambda(g_{\hat{u}}^{-1}A_{g_{\hat{u}}})\in\partial\Gamma_k^+$ in the viscosity sense on $M^n$, then there is no smooth metric $g\in[g_0]$ satisfying $\lambda(g^{-1}A_g)\in\Gamma_k^+$ on $M^n$. This fact, combined with the second statement in Theorem \ref{f'}, yields the following:
	
	\begin{cor}\label{504}
		Let $(M^n,g_0)$ be a smooth, closed Riemannian manifold of dimension $n\geq 3$ and let $2\leq k \leq n$. Suppose there exists a metric $g_{\hat{u}} = e^{-2\hat{u}}g_0$, $\hat{u}\in C^0(M^n)$, satisfying $\lambda(g_{\hat{u}}^{-1}A_{g_{\hat{u}}})\in\partial\Gamma_k^+$ in the viscosity sense on $M^n$. Then there exists a $C^{1,1}$ metric $g=e^{-2u}g_0$ satisfying $\lambda(g^{-1}A_g)\in\partial\Gamma_k^+$ a.e.~on $M^n$. 
	\end{cor}
	
	\begin{rmk}\label{504''}
		It is an open problem as to whether all continuous viscosity solutions $g=e^{-2u}g_0$ to $\lambda(g^{-1}A_g)\in\partial\Gamma_k^+$ belong to $C^{1,1}$, and if solutions $u$ are unique up to addition of constants. The question of uniqueness is related to the known failure of the strong comparison principle for sub/supersolutions to the equation $\lambda(g^{-1}A_g)\in\partial\Gamma_k^+$ -- see \cite{LN09}. We note that, for $t<1$, uniqueness and $C^{1,1}$ regularity of continuous viscosity solutions to $\lambda(g^{-1}A_g^t)\in\partial\Gamma_k^+$ follow from Corollary \ref{505'} and Proposition \ref{60}. 
	\end{rmk}

	In the proof of Theorem \ref{f'} we will replace $\Gamma_k^+$ with more general cones $\Gamma$ satisfying the following properties (which we assume throughout the paper): 
	\begin{align}
	& \Gamma\!\subset\!\mathbb{R}^n \text{ is an open, convex, connected symmetric cone with vertex at the origin}\label{A} \\
	& \Gamma_n^+ \subseteq \Gamma \subseteq \Gamma_1^+. \label{B}
	\end{align}
	
	\noindent More precisely, for fixed $t\leq 1$ we are concerned with the existence of smooth conformal metrics $g\in[g_0]$ satisfying 
	\begin{equation}\label{79}
	\lambda(g^{-1}A_g^t)\in\Gamma \quad\text{on }M^n,
	\end{equation}
	under the assumption that there exists a metric $g_{\hat{u}} = e^{-2\hat{u}}g_0$, $\hat{u}\in C^0(M^n)$, satisfying 
	\begin{equation}\label{40}
	\lambda(g_{\hat{u}}^{-1}A_{g_{\hat{u}}}) \in \overline{\Gamma}\quad \text{in the viscosity sense on }M^n. 
	\end{equation} 
	Equivalently\footnote{The equivalence \eqref{79} and \eqref{44} is easily seen by observing $A_g^t = \tau^{-1}[\tau A_g + (1-\tau)\sigma_1(\lambda(g^{-1}A_g)) g]$ for $\tau= \tau(t) = (1+\frac{1-t}{n-2})^{-1}$.}, for $\tau\in[0,1]$, $\lambda\in\mathbb{R}^n$ and $e=(1,\dots,1)\in\mathbb{R}^n$, we define
	\begin{equation*}
	\lambda^\tau = \tau\lambda +(1-\tau)\sigma_1(\lambda)e\quad\text{and}\quad \Gamma^\tau = \{\lambda\in\mathbb{R}^n: \lambda^\tau \in\Gamma\},
	\end{equation*} 
	and we consider the existence of smooth conformal metrics $g\in[g_0]$ satisfying 
	\begin{equation}\label{44}
	\lambda(g^{-1}A_{g})\in\Gamma^\tau\quad\text{on }M^n
	\end{equation}
	under \eqref{40}. Clearly, $\Gamma^1 = \Gamma$ and $\Gamma^0 = \Gamma_1^+$. We note that the inclusion $\Gamma\subseteq \Gamma_1^+$ in \eqref{B} implies the monotonicity property $\Gamma^\tau\subseteq\Gamma^{\tau'}$ for $\tau' \leq \tau \leq 1$, and the properties \eqref{A} and \eqref{B} are inherited by $\Gamma^\tau$ for $\tau<1$. 
	
	In fact, under \eqref{40} we will obtain existence and uniqueness results for a class of fully nonlinear equations, whose solutions satisfy the differential inclusion \eqref{44}. To this end, for $\Gamma$ as above, suppose that $f:\mathbb{R}^n\rightarrow\mathbb{R}$ satisfies the following properties:
	\begin{align} 
	& f\in C^\infty(\Gamma)\cap C^0(\overline{\Gamma}) \text{ is concave, $1$-homogeneous and symmetric in the }\lambda_i,\label{C} \\[6pt]
	& f>0 \text{ in }\Gamma, \quad f=0 \text{ on }\partial\Gamma \quad\text{and}\quad  f_{\lambda_i}>0 \text{ in }\Gamma \text{ for } 1 \leq i\leq n.\label{D}
	\end{align}
	(For the existence of such a defining function for $\Gamma$ satisfying \eqref{A} and \eqref{B}, see \cite[Appendix A]{LN20}). Of particular interest are equations of the form
	\begin{equation}\label{2-}
	f^\tau\big(\lambda(g_0^{-1}A_{g_{u_\tau}})\big)= \mu e^{au_\tau}, \quad \lambda(g_0^{-1}A_{g_{u_\tau}})\in\Gamma^\tau, 
	\end{equation}
	where 
	\begin{equation*}
	f^\tau(\lambda) = f(\lambda^\tau),
	\end{equation*}
	$\mu>0$ is a constant, $u_\tau$ is the unknown function, $g_{u_\tau} = e^{-2u_\tau}g_0$, and either $a=0$ or $a=-2$. When $a=-2$, \eqref{2-} is critical and includes the (trace-modified) $\sigma_k$-Yamabe equation as a special case (namely when $f=\sigma^{1/k}_k$). In this paper, we consider the case $a=0$. In this case, both the determination of the function $u_\tau$ and the constant $\mu$ are part of the existence problem for \eqref{2-}, and \eqref{2-} is often referred to as a \textit{nonlinear eigenvalue problem}. Since the work of Lions \cite{L85} and Lions, Trudinger \& Urbas \cite{LTU86}, fully nonlinear elliptic eigenvalue problems (including for the $\sigma_k$-Yamabe operator) have attracted much interest -- for a partial list of references, see \cite{QS08, Arm09, EFQ10, GLW10, Sir10} and the references therein. 
	
	\begin{rmk}
		We note that when the pair $(f,\Gamma)$ satisfies the properties \eqref{A}, \eqref{B}, \eqref{C} and \eqref{D}, so too does the pair $(f^\tau, \Gamma^\tau)$ for $\tau\in[0,1]$. Moreover, whenever $\tau<1$, the pair $(f^\tau,\overline{\Gamma^\tau})$ is locally strictly elliptic (see e.g.~\cite[equation (A.1)]{LN20}).
	\end{rmk} 
	
	Since we will only be concerned with \eqref{2-} in the case $a=0$, we state this equation separately for later reference:
	\begin{equation}\label{5}
	f^\tau(\lambda(g_0^{-1}A_{g_{u_\tau}}))= \mu_\tau, \quad \lambda(g_0^{-1}A_{g_{u_\tau}})\in\Gamma^\tau.
	\end{equation}
	Note that we have denoted the constant in \eqref{5} by $\mu_\tau$ rather than $\mu$; this reflects the fact the constant $\mu_\tau>0$ for which \eqref{5} admits a solution depends uniquely on $\tau$. Indeed, we prove the following existence and uniqueness result for \eqref{5}:

	\begin{thm}\label{e}
		Let $(M^n,g_0)$ be a smooth, closed Riemannian manifold of dimension $n\geq 3$ and let $(f,\Gamma)$ satisfy \eqref{A}, \eqref{B}, \eqref{C} and \eqref{D} with $\Gamma \not= \Gamma_1^+$. Suppose that there exists a metric $g_{\hat{u}} = e^{-2\hat{u}}g_0$, $\hat{u}\in C^0(M^n)$, satisfying \eqref{40}. Then the following statements hold: \medskip 
		\begin{enumerate}
			\item[1.] For given $\tau\in(0,1)$, there exists a constant $\mu_\tau>0$ and $u_\tau\in C^\infty(M^n)$ satisfying \eqref{5} if and only if $\hat{u}$ is not a smooth solution to $\operatorname{Ric}_{g_{\hat{u}}} \equiv 0$ on $M^n$.  \medskip
			
			\item[2.] There exists a constant $\mu_\tau>0$ and $u_\tau\in C^\infty(M^n)$ satisfying \eqref{5} for each $\tau\in(0,1]$ if and only if there is no $C^{1,1}$ metric $g$ conformal to $g_0$ satisfying $\lambda(g^{-1}A_g)\in\partial\Gamma$ a.e.~in $M^n$. \medskip
		\end{enumerate}
		Moreover, the eigenfunction/eigenvalue pair is unique in the following sense: if $(u_\tau,\mu_\tau)$, $(\check{u}_\tau,\check{\mu}_\tau) \in C^\infty(M^n)\times (0,\infty)$ both satisfy \eqref{5}, then $\check{\mu}_\tau = \mu_\tau$ and $\check{u}_\tau=u_\tau+c$ for some constant $c\in\mathbb{R}$. 
	\end{thm}
	
	\begin{rmk}
		Suppose $(f,\Gamma) = (\sigma_2^{1/2}, \Gamma_2^+)$ and that there exists a constant $\mu_1>0$ and $u_1\in C^\infty(M^n)$ satisfying \eqref{5} with $\tau=1$. Then by uniqueness in Theorem \ref{e}, the constant $\mu_1$ coincides with the constant $\lambda(g_0,\sigma_2)^{1/2}$ in \cite[Theorem 1]{GLW10}. See Section \ref{g2} for the definition of $\lambda(g_0,\sigma_2)$. 
	\end{rmk}

	We now highlight some elements of the proof of Theorem \ref{e}. For simplicity, we address only the first statement in Theorem \ref{e}. We first show in Section \ref{100} that \eqref{40} and $\Gamma\not=\Gamma_1^+$ imply either $Y(M^n,[g_0])>0$ or $\hat{u}$ is smooth with $\operatorname{Ric}_{g_{\hat{u}}}\equiv 0$. In the latter case we are done. In the former case, rather than directly addressing \eqref{5}, we first consider a class of equations for which we prove both existence and uniqueness under \eqref{40}. More precisely, for $\tau\in(0,1)$ and $h=h(x,z)>0$ satisfying conditions (C1)-(C3) in Section \ref{est}, we prove existence and uniqueness of solutions to
	\begin{equation}\label{2}
	f^\tau\big(\lambda(g_0^{-1}A_{g_{u_\tau}})\big)= h(x,u_\tau), \quad \lambda(g_0^{-1}A_{g_{u_\tau}})\in\Gamma^\tau
	\end{equation}
	assuming \eqref{40} and $Y(M^n,[g_0])>0$ -- see Theorem \ref{a} in Section \ref{est}. For each $\tau\in(0,1)$, the solution in the first statement in Theorem \ref{e} is then obtained as a suitably rescaled limit of solutions to \eqref{2} in the case $h(x,z)=e^{\beta z}$ as $\beta\rightarrow 0^+$. Now, the uniqueness for \eqref{2} follows from (C1) (which is a properness condition on $h$) and the strong comparison principle. For existence we use a degree argument, which relies on obtaining \textit{a priori} $C^2$ estimates on solutions $u_\tau$ which are uniform with respect to $\tau$ on any compact subset of $(0,1)$. These are obtained in Section \ref{101}, where the main task is to establish a lower $C^0$ bound (the upper $C^0$ bound follows a standard argument, and first/second derivative estimates follow from previous work of various authors). We assume for a contradiction that the lower $C^0$ bound fails along a sequence $\tau_i\rightarrow \tau<1$, which we show implies the existence of a $C^{1,1}$ metric $\widetilde{g} = e^{-2\widetilde{u}}g_0$ satisfying $\lambda(\widetilde{g}^{-1}A_{\widetilde{g}})\in\partial\Gamma^\tau$ a.e.~on $M^n$. We now recall the assumption that $\hat{u}\in C^0(M^n)$ satisfies \eqref{40}, i.e.~$\lambda(g_{\hat{u}}^{-1}A_{g_{\hat{u}}})\in\overline{\Gamma}$ in the viscosity sense on $M^n$. At this point, if one could show that $\widetilde{u} = \hat{u} + c$ for some constant $c$, a contradiction would follow from a geometric property of the cone $\Gamma$ (see Lemma \ref{506}). For this purpose, we prove a strong comparison principle for $\tau<1$ in Section \ref{scp} (see Theorem \ref{t20}).

	We note that the analogous strong comparison principle for $\tau=1$ is false in general -- see \cite{LN09}. Moreover, for cones $\Gamma$ where such a comparison principle can be established, the second statements in each of Theorems \ref{f'} and \ref{e} can be improved. This is the case, for example, when $(1,0,\dots,0)\in\Gamma$, due to our strong comparison principle (see Theorem \ref{t20}) and the fact that such $\Gamma$ can be written as $(\widetilde{\Gamma})^\tau$ for some $\tau<1$ and $\widetilde{\Gamma}$ satisfying \eqref{A} and \eqref{B} (see Proposition \ref{60} in the appendix). We have:

	\begin{thm}\label{e''}
		In addition to the hypotheses of Theorem \ref{e}, suppose also that \linebreak $(1,0,\dots,0)\in\Gamma$ and $g_{\hat{u}}$ is not a solution to $\lambda(g_{\hat{u}}^{-1}A_{g_{\hat{u}}})\in\partial\Gamma$ on $M^n$ in the viscosity sense. Then for all $\tau\in(0,1]$ there exists a constant $\mu_\tau>0$ and $u_\tau\in C^\infty(M^n)$ satisfying \eqref{5}. In particular, there exists a smooth metric $g\in[g_0]$ such that $\lambda(g^{-1}A_g)\in\Gamma$ on $M^n$. Moreover, if $(u_\tau,\mu_\tau),(\check{u}_\tau,\check{\mu}_\tau)\in C^\infty(M^n)\times (0,\infty)$ both satisfy \eqref{5}, then $\check{\mu}_\tau = \mu_\tau$ and $\check{u}_\tau=u_\tau+c$ for some constant $c\in\mathbb{R}$.
	\end{thm}
	
	As a consequence we have the following uniqueness and regularity result:
	
	\begin{cor}\label{505'}
		Let $(M^n,g_0)$ be a smooth, closed Riemannian manifold of dimension $n\geq 3$ and suppose $\Gamma$ satisfies $(1,0,\dots,0)\in\Gamma$, \eqref{A} and \eqref{B} with $\Gamma\not=\Gamma_1^+$. Then continuous viscosity solutions $g_{u}=e^{-2u}g_0$ to the equation $\lambda(g_u^{-1}A_{g_u})\in\partial\Gamma$ on $M^n$ are unique up to addition of constants to $u$, and belong to $C^{1,1}(M^n)$. 
	\end{cor}

	\begin{rmk}\label{505''}
		By Corollary \ref{505'}, it is equivalent to assume in Theorem \ref{e''} that $g_{\hat{u}}$ is not a $C^{1,1}$ solution to $\lambda(g_{\hat{u}}^{-1}A_{g_{\hat{u}}})\in\partial\Gamma$ a.e.~on $M^n$. 
	\end{rmk}

	\subsection{Applications}
	
	We now discuss some geometric applications of our main results. We start with two results that are consequences of Theorem \ref{f'} (more precisely, Corollary \ref{504'}) and previous work of Ge, Lin \& Wang \cite{GLW10}, and concern the case  $(f,\Gamma) = (\sigma_2^{1/2},\Gamma_2^+)$. We refer the reader to Section \ref{g2} for the definition of the nonlinear eigenvalue $\lambda(\sigma_2,g_0)$ and the nonlinear Yamabe-type invariant $Y_{2,1}([g_0])$, which were previously studied in \cite{GLW10} (see also \cite{GW04, GW06, GW07}).

	\begin{thm}\label{61'}
		Let $(M^3,g_0)$ be a smooth, closed Riemannian 3-manifold. Then \linebreak $\lambda(\sigma_2,g_0)$ is positive (resp.~negative, resp.~zero) if and only if $Y_{2,1}([g_0])$ is positive (resp. negative, resp.~zero). In particular, the sign of $\lambda(\sigma_2,g_0)$ is a conformal invariant.
	\end{thm}

\noindent 	The counterpart to Theorem \ref{61'} in dimensions $n\geq 4$ was previously obtained \cite{GLW10}.

	The relevance of Theorem \ref{61'} (and its higher dimensional counterpart in \cite{GLW10}) in relation to the $\sigma_2$-Yamabe problem is as follows. It is shown in \cite[Theorem 2]{GLW10} that $Y_{2,1}([g_0])>0$ if and only if there exists $g\in[g_0]$ with $\lambda(g^{-1}A_g)\in\Gamma_2^+$ on $M^n$. Given the existence of a solution to the $\sigma_2$-Yamabe problem under the assumption of a conformal metric satisfying $\lambda(g^{-1}A_g)\in\Gamma_2^+$, the existence of a solution $\sigma_2$-Yamabe problem in $[g_0]$ is therefore equivalent to positivity of $\lambda(\sigma_2,g_0)$. This is in direct analogy with the Yamabe problem, where the existence of a conformal metric with positive constant scalar curvature is equivalent to positivity of the first eigenvalue of the conformal Laplacian of any conformal metric. 
	
	Our next result both unifies and extends the work of Ge, Lin \& Wang \cite[Theorem 2]{GLW10} and Catino \& Djadli \cite[Theorem 1.3]{CD10}, as we will explain in more detail in Section \ref{g2}:
	
	\begin{thm}\label{d'}
		Let $(M^n,g_0)$ be a smooth, closed Riemannian manifold of dimension $n\geq 3$ with $Y_{2,1}([g_0]) = 0$. Then there exists a smooth metric $g_t\in[g_0]$ satisfying $\lambda(g_t^{-1}A_{g_t}^t)\in\Gamma_2^+$ if and only if $t<1$. 
	\end{thm}

		Theorem \ref{d'} also has geometric consequences regarding the existence of conformal metrics with pinched Ricci curvature -- see Corollary \ref{23}.

	Our next two results are applications of Theorem \ref{e''}. The first of these applications generalises a result of Aubin \& Ehrlick \cite{Aub70, Ehr76}, who showed that if a manifold admits a smooth metric $g_0$ with $\operatorname{Ric}_{g_0}\geq 0$ on $M^n$ and $\operatorname{Ric}_{g_0}>0$ at some point on $M^n$, then it admits a conformal metric $g\in[g_0]$ with $\operatorname{Ric}_g>0$ on $M^n$. We extend this result to a wider range of lower bounds, which are only required to be satisfied in the viscosity sense on $M^n$: 
	
	\begin{thm}\label{700}
		Let $(M^n,g_0)$ be a smooth, closed Riemannian manifold of dimension $n\geq 3$. Suppose $\hat{g}$ is a continuous metric conformal to $g_0$ such that 
		\begin{equation}\label{312}
		\operatorname{Ric}_{\hat{g}} \geq \alpha R_{\hat{g}}\hat{g} \quad \text{ in the viscosity sense on }M^n
		\end{equation}
		for some constant $\alpha< \frac{1}{2(n-1)}$, and that $\hat{g}$ is not a $C^{1,1}$ solution to $\lambda(\hat{g}^{-1}(\operatorname{Ric}_{\hat{g}} - \alpha R_{\hat{g}}\hat{g})\in\partial\Gamma_n^+$ a.e.~on $M^n$. Then there exists a smooth metric $g\in[g_0]$ such that $\operatorname{Ric}_g > \alpha R_g g$ on $M^n$.
	\end{thm}
	
	Note that when $\alpha \geq 0$, all the relevant metrics in Theorem \ref{700} have nonnegative Ricci curvature, and Theorem \ref{700} can also be obtained by appealing to the work of \cite{LN14} (rather than Theorem \ref{e''}) and our strong comparison principle in Section \ref{scp} -- see Section \ref{802} for the details.

	Our second application of Theorem \ref{e''} establishes a partial converse to \cite[Theorem 1.2]{LN20}, where the existence of a Green's function for $\Gamma$ satisfying \eqref{A} and \eqref{B} is obtained under the assumption that there exists a smooth metric $g\in[g_0]$ satisfying \eqref{79} with $t=1$ (together with a natural, necessary condition on $\Gamma$ that is not relevant in the discussion below). For $m\geq 1$ distinct points $p_1,\dots,p_m\in M^n$ and positive numbers $c_1,\dots,c_m$, we recall from \cite{LN20} that a positive function $w\in C^0_{\operatorname{loc}}(M^n\backslash\{p_1,\dots,p_m\})$ is a Green's function for $\Gamma$ with poles $p_1,\dots,p_m$ and strengths $c_1,\dots,c_m$ if the metric $g = w^{\frac{4}{n-2}}g_0$ satisfies $\lambda(g^{-1}A_{g})\in\partial\Gamma$ in the viscosity sense on $M^n\backslash\{p_1,\dots,p_m\}$ and $\lim_{x\rightarrow p_i} d_g(x,x_i)^{n-2}w(x) = c_i$ for each $i=1,\dots,m$. We prove:
	
	\begin{thm}\label{305}
		Let $(M^n,g_0)$ be a smooth, closed Riemannian manifold of dimension $n\geq 3$ and suppose $\Gamma$ satisfies \eqref{A}, \eqref{B} and $(1,0,\dots,0)\in\Gamma$. Suppose for some finite set of points $\{p_1,\dots,p_m\}\subset M^n$ there exists a Green's function for $\Gamma$ with poles $p_1,\dots,p_m$. Then there exists a smooth metric $g\in[g_0]$ satisfying $\lambda(g^{-1}A_g)\in\Gamma$ on $M^n$. 
	\end{thm}
	
	\begin{rmk}
		It would be interesting to determine whether Theorem \ref{305} is true without the condition $(1,0,\dots,0)\in\Gamma$, i.e.~assuming $(1,0,\dots,0)\in\partial\Gamma$. 
	\end{rmk}

	The plan of the paper is as follows. In Section \ref{scp} we establish a strong comparison principle, which will be used as a tool in proving our main results. Section \ref{est} is devoted to the proof of Theorems \ref{f'} and \ref{e}. We start by stating an existence and uniqueness result for \eqref{2} -- see Theorem \ref{a}. In Section \ref{100} we show that \eqref{40} and $\Gamma\not=\Gamma_1^+$ imply either $g_{\hat{u}}$ is smooth with $\operatorname{Ric}_{g_{\hat{u}}}\equiv 0$, or $Y(M^n,[g_0])>0$. In Section \ref{101} we obtain \textit{a priori} estimates for \eqref{2}, assuming $Y(M^n,[g_0])>0$. In Section \ref{102} we prove Theorem \ref{a} (Theorem \ref{f'} then follows immediately). In Section \ref{103}, we use Theorem \ref{a} to prove Theorem \ref{e} by a suitable limiting argument. In Section 4 we present refinements of Theorems \ref{e} and \ref{a}, and in particular we prove Theorem \ref{e''}. In Section \ref{g2} we address some geometric consequences of main results -- in particular, we prove Theorems \ref{61'}, \ref{d'}, \ref{700} and \ref{305}.

	\section{A strong comparison principle}\label{scp}
	
	In this section we prove a strong comparison principle for subsolutions and supersolutions to the equation $\lambda(g^{-1}A_g)\in\partial\Gamma$ under the condition that $(1,0,\dots,0)\in\Gamma$ (recall that we always assume $\Gamma$ satisfies \eqref{A} and \eqref{B}). As a consequence, this yields a strong comparison principle for the equation $\lambda(g^{-1}A_g)\in\partial\Gamma^\tau$ for $\tau\in(0,1)$, regardless of whether $(1,0,\dots,0)\in\Gamma$ -- see Proposition \ref{60} in Appendix \ref{appb}. Our strong comparison principle will be used to obtain \textit{a priori} estimates in Section 3 and is also of independent interest. Our discussion will involve semicontinuous functions, although for the purpose of obtaining our main results in this paper, the reader may assume that all involved functions are continuous.

	We begin by recalling the notion of viscosity subsolutions and viscosity supersolutions, which have appeared extensively in the literature. In what follows, for a domain $\Omega\subset M^n$, we denote by $\operatorname{USC}(\Omega)$ the set of upper semicontinuous functions $u:\Omega\rightarrow \mathbb{R}\cup\{-\infty\}$, and $\operatorname{LSC}(\Omega)$ the set of lower semicontinuous functions $u:\Omega\rightarrow \mathbb{R}\cup\{+\infty\}$. 
	
	For the convenience of the reader, we recall the following expression for $A_{g_u}$: 
	
	\begin{equation}\label{501}
	A_{g_u} = \nabla_{g_0}^2 u  - \frac{1}{2}|du|_{g_0}^2g_0 + du\otimes du + A_{g_0}. 
	\end{equation}
	
	\begin{defn}\label{311}
		Suppose $\Gamma$ satisfies \eqref{A} and \eqref{B}, and let $\Omega\subset M^n$ be a domain equipped with a smooth Riemannian metric $g_0$. For $u\in \operatorname{USC}(\Omega)$, we say that $g_u=e^{-2u}g_0$ satisfies
		\begin{equation}\label{80}
		\lambda(g_u^{-1}A_{{g_u}})\in\overline{\Gamma} \text{ in the viscosity sense on }\Omega
		\end{equation}
		if for any $x_0\in\Omega$ and $\phi\in C^2(\Omega)$ satisfying $u(x_0)=\phi(x_0)$ and $u(x) \leq \phi (x)$ near $x_0$, there holds
		\begin{equation*}
		\lambda(g_\phi^{-1}A_{g_\phi})(x_0)\in\overline{\Gamma}. 
		\end{equation*} 
		For $u\in \operatorname{LSC}(\Omega)$, we say that $g_u=e^{-2u}g_0$ satisfies 
		\begin{equation}\label{81}
		\lambda(g_u^{-1}A_{{g_u}})\in\mathbb{R}^n\backslash\Gamma \text{ in the viscosity sense on }\Omega
		\end{equation}
		if for any $x_0\in \Omega$ and $\phi\in C^2(\Omega)$ satisfying $u(x_0) = \phi(x_0)$ and $u(x) \geq \phi(x)$ near $x_0$, there holds 
		\begin{equation*}
		\lambda(g_\phi^{-1}A_{g_\phi})(x_0)\in \mathbb{R}^n\backslash \Gamma.
		\end{equation*}
		We to refer to an upper semicontinuous (resp.~lower semicontinuous) function $u$ satisfying \eqref{80} (resp.~\eqref{81}) as a \textit{viscosity subsolution} (resp.~\textit{viscosity supersolution}) of the equation $\lambda(g_u^{-1}A_{g_u})\in\partial\Gamma$ on $\Omega$. We call $u\in C^0(\Omega)$ a \textit{viscosity solution} of the equation $\lambda(g_u^{-1}A_{g_u})\in\partial\Gamma$ if it is both a viscosity subsolution and a viscosity supersolution.
	\end{defn}

	\begin{rmk}\label{84}
		If $u\in C^{1,1}_{\operatorname{loc}}(\Omega)$, then \eqref{80} (resp.~\eqref{81}) is equivalent to $\lambda(g_u^{-1}A_{g_u})\in\overline{\Gamma}$ (resp.~$\lambda(g_u^{-1}A_{g_u})\in\mathbb{R}^n\backslash{\Gamma}$) a.e.~in $\Omega$. Therefore, if $u\in C^{1,1}_{\operatorname{loc}}(\Omega)$, then $u$ is a viscosity solution to $\lambda(g_u^{-1}A_{g_u})\in\partial\Gamma$ if and only if $\lambda(g_u^{-1}A_{g_u})\in\partial\Gamma$ a.e.~in $\Omega$. We refer to e.g.~\cite[Lemma 2.5]{LNW20a} for a proof of these facts. 
	\end{rmk}
	
	Our main result in this section is the following strong comparison principle:
	
	\begin{thm}\label{t20}
		Let $\Omega\subset M^n$ be a domain equipped with a smooth Riemannian metric $g_0$ and let $\Gamma$ satisfy \eqref{A} and \eqref{B}. Suppose that for some $\tau\in(0,1)$, $u_1\in \operatorname{USC}(\Omega)$ satisfies 
		\begin{equation}\label{50}
		\lambda(g_{u_1}^{-1}A_{g_{u_1}})\in\overline{\Gamma^\tau} \quad\mathrm{in~the~viscosity~sense~on~}\Omega
		\end{equation}
		and $u_2\in C_{\operatorname{loc}}^{1,1}(\Omega)$ satisfies 
		\begin{equation}\label{51}
		\lambda(g_{u_2}^{-1}A_{g_{u_2}})\in\mathbb{R}^n\backslash{\Gamma^\tau} \quad\mathrm{a.e.~in~}\Omega.
		\end{equation}
		If $u_1 \leq u_2$ in $\Omega$, then either $u_1\equiv u_2$ in $\Omega$ or $u_1<u_2$ in $\Omega$. 
	\end{thm}
	
	\begin{rmk}
		As noted in the introduction, Theorem \ref{t20} does not hold for $\tau=1$ \cite{LN09}. 
	\end{rmk}
	
	We show in Appendix \ref{appb} that if $\Gamma$ satisfies \eqref{A} and \eqref{B}, then $(1,0,\dots,0)\in\Gamma$ if and only if there exists some $\widetilde{\Gamma}\subset\Gamma$ satisfying \eqref{A} and \eqref{B} and a number $\tau<1$ such that $\Gamma = (\widetilde{\Gamma})^\tau$. Thus, we have the following equivalent formulation of Theorem \ref{t20}:
	
	\begin{customthm}{\ref{t20}$'$}
		\textit{Let $\Omega\subset M^n$ be a domain equipped with a smooth Riemannian metric $g_0$ and let $\Gamma$ satisfy \eqref{A}, \eqref{B} and $(1,0,\dots,0)\in\Gamma$. Suppose that $u_1\in \operatorname{USC}(\Omega)$ satisfies 
			\begin{equation}
			\lambda(g_{u_1}^{-1}A_{g_{u_1}})\in\overline{\Gamma} \quad\mathrm{in~the~viscosity~sense~on~}\Omega
			\end{equation}
			and $u_2\in C_{\operatorname{loc}}^{1,1}(\Omega)$ satisfies 
			\begin{equation}
			\lambda(g_{u_2}^{-1}A_{g_{u_2}})\in\mathbb{R}^n\backslash\Gamma \quad\mathrm{a.e.~in~}\Omega.
			\end{equation}
			If $u_1 \leq u_2$ in $\Omega$, then either $u_1\equiv u_2$ in $\Omega$ or $u_1<u_2$ in $\Omega$. }
	\end{customthm}

	Our proof of Theorem \ref{t20} is inspired by that of Li, Nguyen \& Wang \cite[Theorem 2.3]{LNW20a}, which in turn uses ideas from \cite{CLN13} and \cite{LNW18}. In \cite{LNW20a}, the authors obtain a strong comparison principle for fully nonlinear, locally strictly elliptic equations satisfying a non-degeneracy condition (see equation (2.3) therein). In our setting, the non-degeneracy condition of \cite{LNW20a} is not satisfied.

	\subsection{A preliminary strong comparison principle}

	We begin with a preliminary comparison principle, analogous to \cite[Proposition 2.6]{LNW20a}. We note that whilst a counterpart to \cite[Proposition 2.7]{LNW20a} is also possible in our setting, it will not be needed in this paper.

	\begin{lem}\label{t18}
		Let $\Omega\subset \mathbb{R}^n$ be an open Euclidean ball, $g_0$ a smooth Riemannian metric on $\overline{\Omega}$, and suppose $\Gamma$ satisfies \eqref{A} and \eqref{B}. Suppose also that $u_1\in \operatorname{USC}(\overline{\Omega})$ satisfies
		\begin{equation}\label{50'}
		\lambda(g_{u_1}^{-1}A_{g_{u_1}})\in\overline{\Gamma} \quad\mathrm{in~the~viscosity~sense~on~}\Omega
		\end{equation}
		and that there exists a constant $\delta>0$ such that $u_2\in C_{\operatorname{loc}}^{1,1}(\Omega)\cap \operatorname{LSC}(\overline{\Omega})$ satisfies 
		\begin{equation}\label{12}
		\lambda(g_{u_2}^{-1}A_{g_{u_2}} + \delta I)\in\mathbb{R}^n\backslash\Gamma \quad \mathrm{~a.e.~in~}\Omega.
		\end{equation}
		If $u_1 \leq u_2$ in $\Omega$ and $u_1<u_2$ on $\partial\Omega$, then $u_1<u_2$ in $\Omega$. 
	\end{lem}

	We establish some notation before proving Lemma \ref{t18}. Let $(x,z,p,B)\in \Omega \times \mathbb{R}\times \mathbb{R}^n \times \operatorname{Sym}_n(\mathbb{R})$. In what follows, for shorthand we write
	\begin{equation}\label{70}
	\lambda(x,z,p,B)\in\overline{\Gamma}\quad\text{if}\quad \lambda\bigg[g_0^{-1}(x)\bigg(B - \frac{1}{2}|p|_{g_0(x)}^2g_0(x) + p\otimes p + A_{g_0(x)}(x)\bigg)\bigg]\in\overline{\Gamma}. 
	\end{equation}
	Note that when $z=u(x)$, $p = du(x)$ and $B=\nabla_{g_0(x)}^2 u(x)$, the quantity inside the square parentheses in \eqref{70} is the $(1,1)$-Schouten tensor of the metric $e^{-2u}g_0$ at $x$, and $(z,p,B) = (u(x),du(x),\nabla_{g_0(x)}^2 u(x))$ is the second order jet of $u$ at $x$, henceforth denoted by $J^2_{g_0}[u](x)$.
	
	The proof of Lemma \ref{t18} is an adaptation of the proof of \cite[Proposition 2.6]{LNW20a}. To aid the exposition, we first prove a weaker version of Lemma \ref{t18} in which we assume $u_1\in C^0(\overline{\Omega})$:
	
	\begin{customlem}{\ref{t18}$'$}\label{t18'}
		\textit{Lemma \ref{t18} holds under the stronger assumption that $u_1\in C^0(\overline{\Omega})$. }
	\end{customlem}

	\begin{proof}[Proof of Lemma \ref{t18'}]
		We suppose for a contradiction that there is a point $\hat{x}\in\Omega$ such that $u_1(\hat{x}) = u_2(\hat{x})$, and we let $\Omega'\Subset\Omega$ be a ball concentric to $\Omega$ such that $\hat{x}\in \Omega'$ and $u_1<u_2$ on $\partial\Omega'$. For $\epsilon>0$ small and $x\in \Omega'$, we denote by $\hat{u}_\ep(x)$ the $\sup$-convolution of $u_1$ with respect to $g_0$:
		\begin{equation}\label{71}
		\hat{u}_\ep(x) \defeq \sup_{y\in\Omega}\big(u_1(y) - \ep^{-1}d_{g_0}(x,y)^2\big) \geq u_1(x). 
		\end{equation}
		It is well-known that for $\epsilon$ sufficiently small, the supremum in \eqref{71} is attained in $\Omega$, namely for each $x\in \Omega'$ there exists a point $x^* = x^*(\ep, x)\in\Omega$ such that 
		\begin{equation*}
		\hat{u}_\ep(x) = u_1(x^*) - \ep^{-1}d_{g_0}(x,x^*)^2.
		\end{equation*}
		Moreover, $\hat{u}_\ep$ is punctually second order differentiable (see \cite{CC95} for the definition) a.e.~in $\Omega'$ with $\nabla^2_{g_0}\hat{u}_\ep \geq -C(\Omega',g_0)\ep^{-1}g_0$ a.e.~in $\Omega'$, and $\hat{u}_\ep$ converges monotonically to $u_1$ as $\ep\rightarrow 0$. We will also use the fact that, if $m:[0,\infty)\rightarrow [0,\infty)$ is a modulus of continuity for $u_1$ in $\Omega'$, i.e.~$m(0)=0$ and $|u_1(p) - u_1(q)| \leq m(d_{g_0}(p,q))$ for all $p,q\in\overline{\Omega'}$, then
		\begin{equation}\label{115}
		\ep^{-1}d_{g_0}(x,x^*)^2 \leq m\bigg(\Big(C(\Omega',g_0)\ep \sup_{\overline{\Omega'}}|u_1|\Big)^{1/2}\bigg)
		\end{equation}
		for all $x\in\Omega'$, from which it follows that 
		\begin{equation}\label{500}
		\lim_{\ep\rightarrow 0} \ep^{-1}d_{g_0}(x,x^*)^2 = 0 \quad\text{for all }x\in\Omega'.
		\end{equation}
		For a proof of these facts, see \cite{CC95, LNW18}. \medskip

		\noindent\textbf{Claim:} \textit{There exists a sequence $(\ep, \eta) \rightarrow (0,0)$, a corresponding sequence of points $\{y_{\ep,\eta}\} \subset \Omega'$, and a constant $C_1$ independent of $\epsilon$ and $\eta$ such that \medskip 
			\begin{enumerate}
				\item$\hat{u}_\ep$ and $u_2$ are punctually second order differentiable at $y_{\ep,\eta}$, \medskip 
				\item for $y_{\ep,\eta}^*  = x^*(\ep, y_{\ep,\eta})$,
				\begin{equation*}
				\lambda \big(y_{\ep,\eta}^*, \hat{u}_\ep(y_{\ep,\eta}) + \ep^{-1}d_{g_0}(y_{\ep,\eta},y_{\ep,\eta}^*)^2, d\hat{u}_\ep(y_{\ep,\eta}),\nabla_{g_0}^2u_2(y_{\ep,\eta}) + C_1\eta g_0\big)\in\overline{\Gamma},
				\end{equation*}
				\item $\hat{u}_\ep (y_{\ep,\eta}) -  u_2 (y_{\ep,\eta})\rightarrow 0 $, and \medskip 
				\item $|d\hat{u}_\ep(y_{\ep,\eta}) - du_2(y_{\ep,\eta})| \leq C_1\eta$. \medskip
		\end{enumerate}}

		Assuming the claim for now, we explain how to obtain the desired contradiction. First observe that since $u_2\in C^{1,1}_{\operatorname{loc}}(\Omega)$, we have $|J^2_{g_0}[u_2](y_{\ep,\eta})| \leq C\|u_2\|_{C^{1,1}(\overline{\Omega'},g_0)}$, where here and below $C$ is a constant independent of $\ep$ and $\eta$, which may change from line to line. Therefore, after restricting to a subsequence of $(\ep,\eta)$ if necessary, there exists $(y_0,J_0)$ for which
		\begin{equation}\label{78'}
		(y_{\ep,\eta},J^2_{g_0}[u_2](y_{\ep,\eta}))\rightarrow (y_0,J_0). 
		\end{equation}
		Then by the third and fourth properties in our claim, \eqref{500} with $x=y_{\ep,\eta}$ and \eqref{78'},
		\begin{equation}\label{112}
		\big(y_{\ep,\eta}^*, \hat{u}_\ep(y_{\ep,\eta}) + \ep^{-1}d_{g_0}(y_{\ep,\eta},y_{\ep,\eta}^*)^2, d\hat{u}_\ep(y_{\ep,\eta}),\nabla_{g_0}^2u_2(y_{\ep,\eta}) + C_1\eta g_0\big)\rightarrow (y_0, J_0). 
		\end{equation}
		
		Appealing to \eqref{112} and the second property in our claim, we have that $\lambda (y_0,J_0)\in\overline{\Gamma}$. Returning to \eqref{78'}, we therefore see that distance from $\lambda (y_{\ep,\eta},J^2_{g_0}[u_2](y_{\ep,\eta}))$ to $\overline{\Gamma}$ tends to zero along a sequence $(\ep,\eta)\rightarrow (0,0)$, contradicting \eqref{12}. 
		
		To complete the proof of Lemma \ref{t18'}, it remains to prove the above claim. \medskip

		\noindent\textit{Proof of the claim.} We first show that if $x\in \Omega'$ is any point where $\hat{u}_\ep$ is punctually second order differentiable, then 
		\begin{equation}\label{76}
		\lambda (x^*, \hat{u}_\ep(x) + \ep^{-1}d_{g_0}(x,x^*)^2,d\hat{u}_\ep(x),\nabla_{g_0}^2\hat{u}_\ep(x)) \in\overline{\Gamma}. 
		\end{equation}
		Indeed, by definition of punctual second order differentiability at $x$, we have
		\begin{equation}\label{72}
		\hat{u}_{\ep}(\operatorname{exp}_x(z)) \leq \hat{u}_\ep(x) + d\hat{u}_\ep(x)(z)+ \frac{1}{2}\nabla^2_{g_0}\hat{u}_\ep(x)(z,z) + o(|z|_{g_0}^2)\quad\text{as }z\rightarrow 0,
		\end{equation}
		where $\operatorname{exp}_x:T_x\Omega\rightarrow \Omega$ is the exponential map at $x$ with respect to $g_0$. On the other hand, by definition of $\hat{u}_\ep$,
		\begin{align}\label{73}
		\hat{u}_\ep(\operatorname{exp}_x(z)) & \geq u_1(\operatorname{exp}_{x^*}(Pz)) - \ep^{-1}d_{g_0}\big(\operatorname{exp}_x(z),\operatorname{exp}_{x^*}(Pz)\big)^2,
		\end{align}
		where $P:T_x \Omega \rightarrow T_{x^*}\Omega$ is the parallel transport map along the unique length-minimising geodesic from $x$ to $x^*$. As shown in the proof of \cite[Proposition 2.4]{LN20}, the first and second variation formulae for length imply 
		\begin{equation}\label{116}
		d_{g_0}\big(\operatorname{exp}_x(z),\operatorname{exp}_{x^*}(Pz)\big)^2  = d_{g_0}(x,x^*)^2 + o(|z|_{g_0}^2)\quad\text{as }z\rightarrow 0.
		\end{equation}
		After substituting \eqref{116} into \eqref{73}, and then \eqref{73} back into \eqref{72}, we obtain
		\begin{equation}\label{74}
		u_1(\operatorname{exp}_{x^*}(Pz)) \leq \hat{u}_\ep(x) + \ep^{-1}d_{g_0}(x,x^*)^2 + d\hat{u}_\ep(x)(z) + \frac{1}{2}\nabla_{g_0}^2\hat{u}_\ep(x)(z,z) + o(|z|_{g_0}^2) \quad\text{as }z\rightarrow 0. 
		\end{equation}
		The inclusion \eqref{76} then follows from \eqref{74} and the fact that $u_1$ is a viscosity subsolution.

		We now construct a sequence of points $y_{\ep,\eta}$ with the properties stated in our claim. For $\ep, \eta >0$, let $\tau=\tau(\ep,\eta)$ be such that 
		\begin{equation}\label{304}
		\eta = \sup_{\Omega'} (\hat{u}_\ep - u_2 + \tau). 
		\end{equation}
		Recalling that $\hat{x}\in \Omega'$ is such that $u_1(\hat{x}) = u_2(\hat{x})$, we have
		\begin{align}\label{104}
		\tau = u_1(\hat{x}) - u_2(\hat{x}) + \tau \leq \hat{u}_\ep(\hat{x}) - u_2(\hat{x}) + \tau \leq \eta,
		\end{align}
		and since $u_1 \leq u_2$ in $\Omega$,
		\begin{equation}\label{114}
		\eta \geq \tau = \eta - \sup_{\Omega'}(\hat{u}_\ep - u_2) \geq \eta - \sup_{\Omega'}(\hat{u}_\ep - u_1). 
		\end{equation}

		Given $\ep>0$ sufficiently small so that $\hat{u}_\ep - u_2<0$ on $\partial \Omega'$, since $\tau\leq \eta$ (by \eqref{104}), we may choose $\eta$ sufficiently small so that
		\begin{equation}\label{105}
		\xi \defeq \hat{u}_\ep - u_2 + \tau < 0 \quad\text{on }\partial \Omega'.
		\end{equation}

		Now let $\xi^+ = \max(\xi, 0)$ and denote by $\Gamma_{\xi^+}$ the concave envelope of $\xi^+$ in $\Omega'$ (with respect to the Euclidean structure). Note that $\xi$ is semi-convex in $\Omega'$ and, by \eqref{105}, $\xi\leq 0$ on $\partial \Omega'$, and hence by the Alexandrov-Bakelman-Pucci estimate in  \cite[Lemma 3.5]{CC95}, the concave envelope $\Gamma_{\xi^+}$ is in $C^{1,1}(\Omega')$ and
		\begin{equation}\label{75}
		\int_{\{\xi=\Gamma_{\xi^+}\}} \det(-\partial^2 \Gamma_{\xi^+}) \geq \frac{1}{C(\Omega')} (\sup_{\Omega'}\xi)^n > 0,
		\end{equation}
		where $\partial^2\Gamma_{\xi^+}$ is the Hessian matrix of second-order partial derivatives of $\Gamma_{\xi^+}$ with respect to the Euclidean coordinate system. Note that positivity in \eqref{75} follows from the fact that $0<\eta = \sup_{\Omega'} \xi$, and it follows from \eqref{75} that the set $\{\xi = \Gamma_{\xi^+}\}$ has positive measure. 
		
		Now, since $\hat{u}_\ep$ and $u_2$ are punctually second order differentiable a.e., for each sufficiently small $\epsilon>0$ and $\eta>0$ as above, there exists some $y=y_{\ep,\eta}\in\{\xi=\Gamma_{\xi^+}\}$ such that $\hat{u}_\ep$ and $u_2$ are punctually second order differentiable at $y$. We will now see that
		\begin{align}
		&0 <  \xi(y)  = \hat{u}_\ep(y) - u_2(y) + \tau \leq \eta, \label{106} \\
		& |d \xi(y)|  = |d \hat{u}_\ep(y) - d u_2(y) | \leq C\eta, \text{ and} \label{110} \\
		&0 \geq \partial^2\xi(y) = \partial^2 \hat{u}_\ep(y) - \partial^2 u_2(y). \label{77}
		\end{align}
		Indeed, to see \eqref{106}, it is enough to observe that if $\xi(y) = 0$, then the concavity and nonnegativity of $\Gamma_{\xi^+}$ would imply that $\Gamma_{\xi^+}\equiv0$, contradicting \eqref{75}. \eqref{110} follows from \eqref{106} and concavity, and \eqref{77} holds since $\xi$ is concave on the set $\{\xi=\Gamma_{\xi^+}\}$.

		Now, the third property in our claim follows from \eqref{106} and the bound for $\tau$ in \eqref{114}, the fourth property is precisely \eqref{110}, and the first property is satisfied by construction. It remains to prove the second property in our claim. Let $\Gamma_{ij}^k$ denote the Christoffel symbols of $g_0$ with respect to the Euclidean coordinate system. By \eqref{77}, 
		\begin{align}
		[\nabla_{g_0}^2 u_2 (y)]_{ij} & = [\partial^2 u_2(y)]_{ij} - \Gamma_{ij}^k(y) \partial_k u_2(y)  \nonumber \\
		& \geq [\partial^2 \hat{u}_\ep(y)]_{ij} - \Gamma_{ij}^k(y) \partial_k u_2(y)   = [\nabla^2_{g_0} \hat{u}_\ep(y)]_{ij} + \Gamma_{ij}^k(y)\big(\partial_k \hat{u}_\ep(y) -\partial_k u_2(y)\big) \nonumber
		\end{align}
		in the sense of matrices, and by using \eqref{110} to estimate $\partial_k \hat{u}_\ep(y) -\partial_k u_2(y)$, we obtain
		\begin{equation}\label{111}
		\nabla_{g_0}^2 u_2 (y) \geq \nabla^2_{g_0} \hat{u}_\ep(y) - C_1\eta g_0
		\end{equation}
		for some constant $C_1$ independent of $\ep,\eta$.

		Let $y^* = x^*(\ep,y)$. By ellipticity, \eqref{76} and \eqref{111}, we therefore have
		\begin{equation*}
		\lambda\big(y^*, \hat{u}_\ep(y) + \ep^{-1}d_{g_0}(y,y^*)^2, d\hat{u}_\ep(y),\nabla_{g_0}^2u_2(y) + C_1\eta g_0\big)\in\overline{\Gamma}, 
		\end{equation*}
		which is precisely the second property in our claim. This completes the proof. 
	\end{proof}

	\begin{proof}[Proof of Lemma \ref{t18}]
		We amend the argument in the proof of Lemma \ref{t18'}, accounting for the lack of a modulus of continuity for $u_1$. In particular, \eqref{500} is in general not true for $u_1\in \operatorname{USC}(\overline{\Omega})$ (for an example of where this limit is not zero, see Section 2 of \cite{LNW18}). 
		
		We first observe that in the proof of Lemma \ref{t18'}, the continuity of $u_1$ is used only once, namely in passing from \eqref{78'} to \eqref{112} -- the rest of the argument remains valid when $u_1\in\operatorname{USC}(\overline{\Omega})$. In passing from \eqref{78'} to \eqref{112} in the case that $u_1\in C^0(\overline{\Omega})$, we used \eqref{500}. In the semicontinuous case, it suffices to show the weaker estimate
		\begin{equation}\label{306}
		\liminf_{\ep\rightarrow 0} \ep^{-1}d_{g_0}(y_{\ep,\eta}, y_{\ep,\eta}^*)^2 \leq \eta 
		\end{equation}
		in order to pass from \eqref{78'} to \eqref{112}, thereby completing the proof.
		
		The inequality in \eqref{306} was previously obtained in the proof of  \cite[Proposition 2.6]{LNW20a} -- we give the argument here for the convenience of the reader. We will use a weaker version of \eqref{500} for $u_1\in \operatorname{USC}(\overline{\Omega})$, namely that
		\begin{equation}
		\ep^{-1}d_{g_0}(x,x^*)^2\leq C\quad\text{for all }x\in\Omega'
		\end{equation}
		(see \cite{CC95, LNW18}). So we may suppose for some fixed $\eta>0$ and some sequence $\epsilon_m\rightarrow 0$ that
		\begin{equation}\label{307}
		\ep_m^{-1}d_{g_0}(y_m, y_m^*)^2 \rightarrow d,
		\end{equation}
		where $y_m \defeq y_{\ep_m, \eta}$. To prove \eqref{306}, we need to show that $d\leq \eta$.

		Let $\tau_m = \tau(\ep_m,\eta)$ be defined as in \eqref{304}. By \eqref{114}, after restricting to a further subsequence if necessary, we have $\tau_m\rightarrow \tau_0\in(-\infty,\eta]$, and also $y_m\rightarrow y_0\in\overline{\Omega'}$. The latter limit, combined with \eqref{307}, also implies $y_m^*\rightarrow y_0$, and hence by upper semicontinuity of $u_1$, 
		\begin{equation}\label{308}
		\limsup_{m\rightarrow \infty} u_1(y_m^*) \leq u_1(y_0). 
		\end{equation}
		Therefore,
		\begin{align}\label{310}
		0 & \leq \limsup_{m\rightarrow \infty} \ep_m^{-1}d_{g_0}(y_m, y_m^*)^2  \stackrel{\eqref{71}}{=} \limsup_{m\rightarrow \infty}\big(u_1(y_m^*) - \hat{u}_{\ep_m}(y_m)\big) \nonumber \\
		& \leftstackrel{\eqref{106}}{\leq} \limsup_{m\rightarrow \infty} \big(u_1(y_m^*) - u_2(y_m) + \tau_m \big) \leftstackrel{\eqref{308}}{\leq} u_1(y_0) - u_2(y_0) + \eta. 
		\end{align}
		But 
		\begin{equation}\label{309}
		u_1(y_0) - u_2(y_0) = \lim_{m\rightarrow\infty} \big(\hat{u}_{\ep_m}(y_0) - u_2(y_0)\big) \leq \lim_{m\rightarrow\infty} \sup_{\Omega'} (\hat{u}_{\ep_m} - u_2)= \sup_{\Omega'} (u_1 - u_2) = 0,
		\end{equation}
		and substituting \eqref{309} into \eqref{310} we see that $d\leq \eta$, as required. 
	\end{proof}
	
	\subsection{Proof of Theorem \ref{t20}}
	
	With Lemma \ref{t18} now established, we prove Theorem \ref{t20}:

	\begin{proof}[Proof of Theorem \ref{t20}]
		We follow \cite{LNW20a} (see also \cite{CLN13, LNW18}), arguing by contradiction. If the conclusion is false, then we can find a closed ball $\bar{B}\subset\Omega$ of some radius $R>0$ for which there exists $\hat{x}\in\partial B$ with
		\begin{equation*}
		u_1 < u_2 \mathrm{~in~}\bar{B}\backslash\{\hat{x}\}\quad \mathrm{and}\quad u_1(\hat{x}) = u_2(\hat{x}). 
		\end{equation*}
		Taking $B$ smaller if necessary, we may assume it is contained inside a normal coordinate chart about its centre, and for the remainder of the proof we implicitly identify $B$ with its image under this chart, assumed to be centred at the origin in $\mathbb{R}^n$.
		
		We will deform $u_2$ into a strict supersolution $\tilde{u}$ (in the sense of \eqref{12}) in some open Euclidean ball $A$ (contained in the image of the aforementioned chart) centred at $\hat{x}$ such that $ u_1 < \tilde{u}$ on $\partial A$ and $\inf_A (\tilde{u}-u_1)=0$. These imply that $\tilde{u} - u_1$ attains its infimum (equal to zero) in $A$, which contradicts the conclusion of Lemma \ref{t18} that $u_1 < \tilde{u}$ in $A$.

		We construct $\tilde{u}$ as follows. Let $\alpha\gg1$ be a constant, and define 
		\begin{align*}
		E(x) & = e^{-\alpha|x|^2},\nonumber \\
		h(x) & = e^{-\alpha|x|^2} - e^{-\alpha R^2},\nonumber \\
		\xi(x) & = \cos (\alpha^{1/2}(x_1-\hat{x}_1)),
		\end{align*}
		where $|x|^2 = x_1^2 + \dots + x_n^2$. Also let $\mu>0$ and $\nu \geq 0$  be constants, and define the following perturbation of $u_2$:
		\begin{equation*}
		\tilde{u}(x) = \tilde{u}_{\mu,\nu}(x) \defeq u_2(x) - \mu\big(h(x)-\nu\big)\xi(x).
		\end{equation*}

		We start by taking a ball $A$ centred at $\hat{x}$ of sufficiently small radius $R_A$ such that $A$ is contained in the image of the aforementioned chart and $|x|$ is uniformly bounded away from 0 on $A$. For reasons that will become clear later, we also assume in what follows that the constants $R_A, \alpha$ and $\mu$ are chosen such that
		\begin{equation}\label{303}
		R_A < \alpha^{-1} \quad\text{and} \quad \mu\alpha E(x) < 1 \text{ for }x\in A. 
		\end{equation}
		Note that the first of these conditions implies that $\xi>\frac{1}{2}$ on $A$ and $|h(x)| = O(E(x))$; here the implicit constant is universal, although throughout the proof we allow our implicit constants to depend on $\|u\|_{C^1(\overline{A})}$.

		Recall that $\lambda(g_u^{-1}A_{g_u})\in\Gamma^\tau$ if and only if $\lambda^{\tau}(g_u^{-1}A_{g_u})\in\Gamma$, and by \eqref{501} we have
		\begin{equation*}
		\lambda^\tau(g_{\tilde{u}}^{-1}A_{g_{\tilde{u}}}) = \lambda\bigg(g_{\tilde{u}}^{-1}\bigg(\tau \nabla_{g_0}^2 {\tilde{u}} + (1-\tau)\Delta_{g_0} \tilde{u} \,g_0 - c_{n,\tau}{|d\tilde{u}|}_{g_0}^2\,g_0 + \tau d\tilde{u}\otimes d\tilde{u} + A_{g_0,\tau }\bigg)\bigg)
		\end{equation*}
		where $c_{n,\tau} = \frac{1}{2}(n-2-(n-3)\tau)$ and 
		\begin{equation}\label{42}
		A_{g_0,\tau} \defeq \tau A_{g_0} + (1-\tau)\sigma_1(g_0^{-1}A_{g_0})g_0.
		\end{equation}
		Likewise, we denote
		\begin{align}\label{t87}
		A_{g_{\tilde{u}},\tau} & \defeq \tau A_{g_{\tilde{u}}} + (1-\tau)\sigma_1(g_{\tilde{u}}^{-1}A_{g_{\tilde{u}}})g_{\tilde{u}} \nonumber \\
		& = \tau \nabla_{g_0}^2 {\tilde{u}} + (1-\tau)\Delta_{g_0} \tilde{u} \,g_0 - c_{n,\tau}{|d\tilde{u}|}_{g_0}^2\,g_0 + \tau d\tilde{u}\otimes d\tilde{u} + A_{g_0,\tau}. 
		\end{align} 
		(Note that we write $\tau\in[0,1]$ as a subscript in \eqref{42} and \eqref{t87}, to avoid confusion with quantities in \eqref{300} and \eqref{301}).

		For the remainder of the proof, computations are carried out at points $x\in A$ where $u_2$ is punctually second order differentiable; since $u_2 \in C^{1,1}_{\operatorname{loc}}(\Omega)$, the set of such points has full measure in $A$. We make the following claim: \medskip
		
		\noindent \textbf{Claim:}  \vspace*{-7mm}\begin{align}\label{t39}
		A_{g_{\tilde{u}},\tau}(x) & =  A_{g_{u_2},\tau}(x) - 4\tau \mu \alpha^2 E(x)\xi(x) x\otimes x - \tau\mu\alpha\nu \xi(x) e_1\otimes e_1  \nonumber \\
		& \quad  - 4(1-\tau)\mu\alpha^2 E(x)\xi(x) |x|^2g_0(x) - (1-\tau)\mu\alpha \nu\xi(x) g_0(x)  \nonumber \\
		& \quad + O\big(\mu\alpha^{3/2} E(x) + \mu\alpha^{1/2}\nu  + \mu^2\alpha\nu^2\big).
		\end{align}

		Assuming \eqref{t39} for now, we complete the proof of Theorem \ref{t20}. First observe that if $\alpha$ is taken sufficiently large, then the $O\big(\mu\alpha^{3/2} E(x)\big)$ terms in \eqref{t39} are absorbed by the strictly negative term $- 4(1-\tau)\mu\alpha^2 E(x)\xi(x) |x|^2g_0(x)$, and the $O(\mu\alpha^{1/2}\nu)$ terms in \eqref{t39} are absorbed by the strictly negative term $- (1-\tau)\mu\alpha \nu\xi(x) g_0(x)$. For $\mu\nu \ll 1$, the $O(\mu^2 \alpha \nu^2)$ terms in \eqref{t39} are also absorbed by the strictly negative term $- (1-\tau)\mu\alpha \nu\xi(x) g_0(x)$. 
		
		For $R_A, \alpha, \mu$ and $\nu$ satisfying the above constraints, after raising an index using $g_{\tilde{u}}^{-1}$, we therefore obtain for some positive function $p(x)=p_{R_A, \alpha,\mu,\nu}(x) \geq C^{-1} > 0$ the following inequality of $(1,1)$-tensors:
		\begin{align}
		g_{\tilde{u}}^{-1}A_{g_{\tilde{u}},\tau}(x)  \leq g_{\tilde{u}}^{-1}A_{g_{u_2},\tau}(x)  - p(x)g_{\tilde{u}}^{-1}g_0(x) = g_{\tilde{u}}^{-1}A_{g_{u_2},\tau}(x)  - p(x)e^{2\tilde{u}}I. \nonumber 
		\end{align}
		Therefore 
		\begin{equation*}
		g_{\tilde{u}}^{-1}A_{g_{\tilde{u},\tau}}(x) + p(x) e^{2\tilde{u}}I \leq g_{\tilde{u}}^{-1}A_{g_{u_2},\tau}(x),
		\end{equation*}
		and it follows from boundedness of $p$ away from zero that for some constant $\delta>0$,
		\begin{equation}\label{82}
		g_{\tilde{u}}^{-1}A_{g_{\tilde{u},\tau}}(x) + \delta I < g_{\tilde{u}}^{-1}A_{g_{u_2},\tau}(x). 
		\end{equation}
		But by \eqref{51} and ellipticity, the inequality \eqref{82} implies 
		\begin{equation}
		\lambda\big(g_{\tilde{u}}^{-1}A_{g_{\tilde{u},\tau}} + \delta I \big)\in \mathbb{R}^n\backslash\Gamma.
		\end{equation}

		To obtain a contradiction to Lemma 2.5, we only need to choose $\nu$ so that
		\begin{equation}\label{502}
		u_1 < \tilde{u}_{\mu,\nu} \mathrm{~on~}\partial A\quad\text{and}\quad \inf_A(\tilde{u}_{\mu,\nu}- u_1) = 0.
		\end{equation}
		Denote $\nu_0 = \sup_A h \geq \sup_{A\cap B} h >0$. Then for sufficiently small $\mu>0$, it is easy to see that $u_1 < \tilde{u}_{\mu,\nu}$ on $\partial A$ for all $0\leq \nu \leq \nu_0$. Moreover $\inf_A(\tilde{u}_{\mu,0} - u_1) \leq 0 \leq \inf_A (\tilde{u}_{\mu,\nu_0} - u_1)$, so a value of $\nu\in[0,\nu_0]$ can be chosen so that $\inf_A(\tilde{u}_{\mu,\nu}- u_1) = 0$.

		To complete the proof of Theorem \ref{t20}, it remains to prove the above claim. \medskip
		
		\noindent\textit{Proof of the claim.} Calculating $\partial_i\tilde{u}$ and $\partial_i\partial_j \tilde{u}$ explicitly, we see
		\begin{align}
		\partial_i \tilde{u}(x) - \partial_i u_2(x) & =  2\mu \alpha E(x)\xi(x) x_i + \mu \alpha^{1/2} (h(x) - \nu)\sin (\alpha^{1/2}(x_1-\hat{x}_1))\delta_{i1}, \label{t32}
		\\ 
		\partial_i\partial_j \tilde{u}(x) - \partial_i\partial_j u_2(x) & =  2\mu \alpha E(x)\xi(x) \big(-2\alpha x_i x_j+ \delta_{ij}\big) \nonumber
		\\
		& \quad  - 4\mu \alpha^{3/2} E(x) \sin (\alpha^{1/2}(x_1-\hat{x}_1))\delta_{i1}x_j \nonumber 
		\\
		& \quad + \mu\alpha (h(x)-\nu)\xi(x)\delta_{i1}\delta_{j1}. \label{t30}
		\end{align}

		\noindent It follows that
		\begin{align}
		(\nabla_{g_0}^2)_{ij}\tilde{u}(x) - (\nabla_{g_0}^2)_{ij}u_2&(x)= \partial_i\partial_j \tilde{u}(x) - \partial_i\partial_j u_2(x)  - \Gamma_{ij}^k(x)\big( \partial_k \tilde{u}(x) -  \partial_k u_2(x)\big) \nonumber \\
		& \,\stackrel{\eqref{t32}}{=}\partial_i\partial_j \tilde{u}(x) - \partial_i\partial_j u_2(x) -  \Gamma_{ij}^k(x)\Big(2\mu \alpha E(x)\xi(x) x_k\nonumber \\
		& \qquad \,\,  + \mu \alpha^{1/2} (h(x) - \nu)\sin (\alpha^{1/2}(x_1-\hat{x}_1))\delta_{k1}\Big) \nonumber \\
		& = \partial_i\partial_j \tilde{u}(x) - \partial_i\partial_j u_2(x)  + O\big(\mu\alpha E(x) + \mu\alpha^{1/2}\nu\big), \nonumber 
		\end{align}
		where $\Gamma_{ij}^k$ are the Christoffel symbols of the metric $g_0$, and we therefore obtain the identity
		\begin{align}\label{t37}
		\tau\nabla_{g_0}^2 \tilde{u}(x) + (1-\tau)\Delta_{g_0} \tilde{u}\, g_0(x)  & =  \tau\nabla_{g_0}^2 u_2(x) + (1-\tau)\Delta_{g_0} u_2(x) \,g_0(x) \nonumber \\
		& \quad - 4\tau \mu \alpha^2 E(x)\xi(x) x\otimes x - \tau\mu\alpha\nu \xi(x) e_1\otimes e_1  \nonumber \\
		& \quad - 4(1-\tau)\mu\alpha^2 E(x)\xi(x) |x|^2g_0(x) - (1-\tau)\mu\alpha \nu\xi(x) g_0(x)  \nonumber \\
		& \quad + O\big(\mu\alpha^{3/2}E(x) + \mu \alpha^{1/2}\nu\big). 
		\end{align}
		
		Next we calculate $- c_{n,\tau}{|d \tilde{u}|}_{g_0}^2\,g_0 + \tau d\tilde{u}\otimes d\tilde{u}$. Expanding out $|d\tilde{u}|_{g_0}^2(x)$ using \eqref{t32}, and using \eqref{303} to assert $\mu^2\alpha^2 E(x)^2 \leq \mu\alpha E(x)$ and $\mu^2\alpha^{3/2}E(x)\nu \leq \mu\alpha^{1/2}\nu$, we see that 
		\begin{align}\label{t38}
		|d\tilde{u}|_{g_0}^2(x) - |du_2|_{g_0}^2(x) & = O\big(\mu\alpha E(x) + \mu\alpha^{1/2}\nu  + \mu^2\alpha\nu^2\big) 
		\end{align}
		and likewise
		\begin{align}\label{t38'}
		(d\tilde{u}\otimes d\tilde{u} - du_2\otimes du_2)(x) & = O\big(\mu\alpha E(x) + \mu\alpha^{1/2}\nu  + \mu^2\alpha\nu^2\big).
		\end{align}

		Combining \eqref{t37}, \eqref{t38} and \eqref{t38'} we obtain \eqref{t39}, completing the proof of the claim and therefore the proof of Theorem \ref{t20}.  \end{proof}

	\section{Proof of Theorem \ref{e} and related results}\label{est}
	
	In this section we prove Theorem \ref{e}, from which Theorem \ref{f'} follows immediately. As discussed in the introduction, a subtlety in the study of the nonlinear eigenvalue problem 
	\begin{equation}\label{5'}
	f^\tau(\lambda(g_0^{-1}A_{g_{u_\tau}}))= \mu_\tau, \quad \lambda(g_0^{-1}A_{g_{u_\tau}})\in\Gamma^\tau
	\end{equation}
	is the non-existence of solutions for all but one value of $\mu_\tau$. In addition, when there exists a solution to \eqref{5'}, it is clear that the solution is not unique. As a means for establishing Theorem \ref{e}, we first consider a class of equations for which we prove both existence and uniqueness. More precisely, we consider 
	\begin{equation}\label{2'}
	f^\tau\big(\lambda(g_0^{-1}A_{g_{u_\tau}})\big)= h(x,u_\tau), \quad \lambda(g_0^{-1}A_{g_{u_\tau}})\in\Gamma^\tau,
	\end{equation}
	where $h=h(x,z):M^n\times \mathbb{R}\rightarrow(0,\infty)$ is any smooth positive function which is strictly proper and satisfies mild growth conditions as $z\rightarrow\pm \infty$: \medskip 
	
	\begin{enumerate}
		\item[(C1)]  $\frac{\partial h}{\partial z}>0$ on $M^n\times \mathbb{R}$, \medskip
		
		\item[(C2)]  $\displaystyle \limsup_{z\rightarrow -\infty} \sup_{x\in M} h(x,z) = 0$ and $\displaystyle  \limsup_{z\rightarrow -\infty}\sup_{x\in M} \big(|\nabla h(x,z)| + |\nabla^2 h(x,z)|\big)< \infty$, \medskip
		
		\item[(C3)]  $\displaystyle \liminf_{z\rightarrow+\infty}\inf_{x\in M} h(x,z) = +\infty$. \medskip
	\end{enumerate}
	
	\noindent The properness condition \hyperref[C1]{(C1)} has been widely used in the context of nonlinear elliptic equations, and is used in our argument to obtain uniqueness of solutions to \eqref{2'}. The growth conditions \hyperref[C2]{(C2)} and \hyperref[C3]{(C3)} are used in our argument to obtain existence of solutions to \eqref{2'}. An example of a function satisfying \hyperref[C1]{(C1)}-\hyperref[C3]{(C3)} is $h(x,z) = \tilde{h}(x) e^{\beta z}$ for $\beta>0$, where $\tilde{h}>0$ is any smooth positive function. 
	
	Our existence and uniqueness result for \eqref{2'} is as follows:
	
	\begin{thm}\label{a}
		Let $(M^n,g_0)$ be a smooth, closed Riemannian manifold of dimension $n\geq 3$, let $h=h(x,z)$ be a smooth positive function satisfying \hyperref[C2]{(C2)} and \hyperref[C3]{(C3)}, and suppose $(f,\Gamma)$ satisfies \eqref{A}, \eqref{B}, \eqref{C} and \eqref{D} with $\Gamma \not= \Gamma_1^+$. Suppose there exists a metric $g_{\hat{u}} = e^{-2\hat{u}}g_0$, $\hat{u}\in C^0(M^n)$, satisfying \eqref{40}. Then the following statements hold: \medskip 
		\begin{enumerate}
			\vspace*{-4mm}\item[1.] For given $\tau\in(0,1)$, there exists a smooth solution $g_{u_\tau}=e^{-2u_\tau}g_0$ to \eqref{2'} if and only if $\hat{u}$ is not a smooth solution to $\operatorname{Ric}_{g_{\hat{u}}} \equiv 0$ on $M^n$. Moreover, for all $0<\delta<T<1$ and $\alpha\in (0,1)$, there exists a constant $C_1>0$ depending only on $n,g_0,\delta, T, \alpha$ and $h$ such that solutions to \eqref{2'} with $\tau\in[\delta,T]$ satisfy
			\begin{equation}\label{14a}
			\| u_\tau \|_{C^{4,\alpha}(M^n,g_0)} \leq C_1. 
			\end{equation}

			\item[2.] There exists a smooth solution $g_{u_\tau}=e^{-2u_\tau}g_0$ to \eqref{2'} for all $\tau\in(0,1]$ if and only if there is no $C^{1,1}$ metric $g$ conformal to $g_0$ satisfying $\lambda(g^{-1}A_g)\in\partial\Gamma$ a.e.~on $M^n$. Moreover, for all $0<\delta < 1$ and $\alpha\in(0,1)$, there exists a constant $C_2>0$ depending only on $n, g_0, \delta, \alpha$ and $h$ such that solutions to \eqref{2'} with $\tau\in[\delta,1]$ satisfy 
			\begin{equation}\label{14a'}
			\| u_\tau \|_{C^{4,\alpha}(M^n,g_0)} \leq C_2. 
			\end{equation}
			\end{enumerate}
		If $h$ also satisfies (C1), then solutions to \eqref{2'} are unique. 
	\end{thm} 
	
	Note that Theorem \ref{f'} may be viewed as a consequence of either Theorem \ref{e} or Theorem \ref{a}.

	We begin in Section \ref{100} by proving that, if $\Gamma\not=\Gamma_1^+$ and there exists a continuous metric conformal to $g_0$ satisfying \eqref{40}, then either $Y(M^n,[g_0])>0$ or $g_{\hat{u}}$ is smooth with $\operatorname{Ric}_{g_{\hat{u}}}\equiv 0$ on $M^n$. In Section \ref{101}, we obtain the \textit{a priori} estimates in the first statement of Theorem \ref{a}. As alluded to in the introduction, this is the point at which we make use of the strong comparison principle obtained in Section \ref{scp}. In Section \ref{102} we complete the proof of Theorem \ref{a} using a degree argument. In Section \ref{103}, we prove Theorem \ref{e} by appealing to Theorem \ref{a} in the special case $h(x,z) = e^{\beta z}$ ($\beta>0$) and applying a limiting argument as $\beta\rightarrow 0^+$. As a by-product of our methods for proving Theorem \ref{a}, we present in Section \ref{p4} a result of Kazdan-Warner type (which will be used later in Section \ref{g2}).

	\subsection{A Yamabe-positive/Ricci-flat dichotomy}\label{100}
	
	An important step in the proof of Theorem \ref{e} is to show that \eqref{40} and the assumption $\Gamma\not = \Gamma_1^+$ imply that either $Y(M^n,[g_0])>0$ or $g_{\hat{u}}$ is smooth with $\operatorname{Ric}_{g_{\hat{u}}}\equiv 0$ on $M^n$:
	
	\begin{prop}\label{F}
		Let $(M^n,g_0)$ be a smooth, closed Riemannian manifold of dimension $n\geq 3$ and suppose $\Gamma$ satisfies \eqref{A} and \eqref{B} with $\Gamma \not= \Gamma_1^+$. Suppose there exists a metric $g_{\hat{u}}= e^{-2\hat{u}}g_0$, $\hat{u}\in C^0(M^n)$, satisfying \eqref{40}. Then either $Y(M^n,[g_0])>0$ or $g_{\hat{u}}$ is smooth with $\operatorname{Ric}_{g_{\hat{u}}} \equiv 0$ on $M^n$.
	\end{prop}
	
	We note that if $\hat{u}\in C^2(M^n)$, then the conclusion of Proposition \ref{F} is clear in view of the following lemma with $t=0$:
	
	\begin{lem}\label{506}
		Suppose that $\Gamma$ satisfies \eqref{A}, \eqref{B} and $\Gamma\not=\Gamma_1^+$. Then for all $t\in[0,1)$, $\partial\Gamma^t \cap \overline{\Gamma}= \{(0,\dots,0)\}$. 
	\end{lem} 
	\begin{proof}
		Suppose for a contradiction that there exists some non-zero $\lambda\in \partial\Gamma^t \cap \overline{\Gamma}$, and let $f$ be a defining function for $\Gamma$ satisfying \eqref{C} and \eqref{D}. Then $f(t\lambda + (1-t)\sigma_1(\lambda)e) = 0$ and $f(\lambda) \geq 0$. By concavity and homogeneity of $f$, it then follows that
		\begin{equation}
		0 = f(t\lambda + (1-t)\sigma_1(\lambda)e) \geq tf(\lambda) + (1-t)\sigma_1(\lambda)f(e) \geq (1-t)\sigma_1(\lambda)f(e).
		\end{equation}
		Since $(1-t)f(e)>0$, this implies $\sigma_1(\lambda)=0$, i.e.~$\lambda\in \partial\Gamma_1^+\cap\overline{\Gamma}$.
		
		Let $P\subset\mathbb{R}^n$ be the polygon with vertices consisting of all permutations of $\lambda$, that is, the convex hull of all permutations of $\lambda$. By symmetry, all vertices of $P$ belong to $\partial\Gamma_1^+\cap\overline{\Gamma}$, and by convexity of $\partial\Gamma_1^+\cap\overline{\Gamma}$, it follows that $P\subseteq \partial\Gamma_1^+\cap\overline{\Gamma}$. Moreover, the barycentre of $P$ is at the origin, and in particular the origin belongs to the interior of $P$. Thus there is a neighbourhood of the origin relative to $\partial\Gamma_1^+$ which is contained in $\partial\Gamma_1^+\cap \overline{\Gamma}$, and by homothety it follows that $\partial\Gamma_1^+ \cap \overline{\Gamma} = \partial\Gamma_1^+$. This implies that $\Gamma = \Gamma_1^+$, a contradiction.
	\end{proof}

	To handle the case that $\hat{u}$ is merely continuous, we first prove the following lemma: 
	
	\begin{lem}\label{86}
		Let $(M^n,g_0)$ be a smooth, closed Riemannian manifold of dimension $n\geq 3$ and suppose $\Gamma$ satisfies \eqref{A} and \eqref{B}. If there exists a metric $g_{\hat{u}}= e^{-2\hat{u}}g_0$, $\hat{u}\in C^0(M^n)$, satisfying \eqref{40}, then $Y(M^n,[g_0]) \geq 0$. 
	\end{lem}

	\begin{proof}
		We use some ideas from \cite{LN20}. We suppose for a contradiction that the Yamabe invariant is negative, so we may assume $R_{g_0}<0$ on $M^n$. We start by writing $g_{\hat{u}}$ in the form $g_{\hat{u}} = w^{\frac{4}{n-2}}g_0$, so that $w \defeq e^{-(n-2)\hat{u}/2}\in C^0(M^n)$ is positive and
		\begin{equation}\label{95}
		-\Delta_{g_0} w+ c_n R_{g_0} w \geq 0 \quad\text{ in the viscosity sense on }M^n,
		\end{equation}
		where $c_n = \frac{n-2}{4(n-1)}$. We claim that \eqref{95} implies $R_{g_{\hat{u}}} \geq 0$ in the distributional sense, i.e.
		\begin{equation}\label{90}
		\int_{M^n} - w\Delta_{g_0}\phi \,dv_{g_0} \geq - \int_{M^n} c_n R_{g_0} w\phi \,dv_{g_0} \quad\text{for all } 0 \leq \phi\in C^\infty(M^n).
		\end{equation}
		Taking $\phi=1$ in \eqref{90} and appealing to negativity of $R_{g_0}$, this yields the desired contradiction.

		By a standard partition of unity argument, to obtain \eqref{90} it suffices to show that for any open set $\Omega\subset M^n$ contained in a single chart,
		\begin{equation}\label{92}
		\int_{\Omega} -w\Delta_{g_0}\phi\,dv_{g_0} \geq - \int_{\Omega} c_n R_{g_0} w\phi\,dv_{g_0} \quad\text{for all }0 \leq \phi\in C^\infty_c(\Omega).
		\end{equation}
		
		Let $\{w_j\}\subset C^\infty(\overline{\Omega})$ be a sequence of smooth functions converging uniformly to $w$ in $\overline{\Omega}$ and such that $w_j < w$ in $\overline{\Omega}$. For each $j$, consider the functional $T_j:H^1(\Omega)\rightarrow \mathbb{R}$ defined by
		\begin{equation*}
		T_j[\rho] = \int_\Omega \big(|\nabla_{g_0} \rho|_{g_0}^2 + c_n R_{g_0}w_j\rho\big)\,dv_{g_0}.
		\end{equation*}
		By the direct method, there exists a unique minimiser of $T_j$ in $H^1(\Omega)$ subject to the constraints
		\begin{equation}\label{88}
		\rho|_{\partial\Omega} = w_j|_{\partial\Omega}\quad\text{and}\quad \rho \geq w_j \text{ in }\Omega. 
		\end{equation}
		Moreover, this minimiser -- henceforth denoted by $\rho_j$ -- satisfies in the weak sense
		\begin{equation}\label{91}
		-\Delta_{g_0}\rho_j \geq -c_nR_{g_0} w_j \quad\text{in }\Omega
		\end{equation}
		and 
		\begin{equation}\label{87}
		\begin{cases}-\Delta_{g_0}\rho_j = -c_nR_{g_0} w_j & \text{in }\{\rho_j > w_j\} \\
		\rho_j = w_j < w & \text{on }\{\rho_j = w_j\}.
		\end{cases}
		\end{equation}
		In particular, by elliptic regularity, $\rho_j$ is smooth in $\{\rho_j > w_j\}$. 
		
		On the other hand, by \eqref{95}, our negativity assumption on $R_{g_0}$ and the fact that $w_j<w$, it follows that $-\Delta_{g_0} w + c_n R_{g_0} w_j \geq 0$ in the viscosity sense. Combining this with \eqref{87} yields
		\begin{equation*}
		\begin{cases}
		\Delta_{g_0}(w- \rho_j) \leq 0 & \text{in the viscosity sense on }\{\rho_j > w_j\} \\
		\rho_j < w & \text{on }\{\rho_j = w_j\},
		\end{cases}
		\end{equation*}
		and by the comparison principle for viscosity solutions (see e.g.~\cite[Corollary 3.7]{CC95}), it follows that $\rho_j\leq w$ on $\{\rho_j> w_j\}$. We also have $\rho_j = w_j < w$ on $\{\rho_j = w_j\}$, and since $\Omega = \{\rho_j \geq w_j\}$ by the constraint in \eqref{88}, it follows that $\rho_j \leq w$ on $\Omega$.
		
		In summary, we have shown $w_j \leq \rho_j \leq w$ in $\Omega$. By uniform convergence of $w_j$ to $w$, it follows that $\rho_j$ converges in $L^\infty(\Omega)$ to $w$. Testing \eqref{91} against a nonnegative test function $\phi\in C_c^\infty(\Omega)$, integrating by parts and taking $j\rightarrow \infty$ then yields \eqref{92}. 
	\end{proof}

	\begin{proof}[Proof of Proposition \ref{F}]
		By Lemma \ref{86} we know $Y(M^n,[g_0])\geq 0$. If $Y(M^n,[g_0])>0$ we are done, so suppose that $Y(M^n,[g_0]) = 0$. We may then assume that $R_{g_0} = 0$. Let $w\in C^0(M^n)$ be such that $g_{\hat{u}}= w^{\frac{4}{n-2}}g_0$, so that $w$ is positive and by \eqref{95}
		\begin{equation}
		-\Delta_{g_0} w  \geq 0 \quad\text{ in the viscosity sense on }M^n. 
		\end{equation}
		By the strong maximum principle for viscosity supersolutions (see \cite[Proposition 4.9]{CC95}), $w$ is locally constant on $M^n$ and hence constant on $M^n$. It follows that $\hat{u}$ is smooth with $R_{g_{\hat{u}}}\equiv 0$, i.e.~$\lambda(g_{\hat{u}}^{-1}A_{g_{\hat{u}}})\in\partial\Gamma_1^+$. But by Lemma \ref{506} with $t=0$, $\lambda(g_{\hat{u}}^{-1}A_{g_{\hat{u}}})\in\partial\Gamma_1^+\cap \overline{\Gamma}$ implies $\lambda(g_{\hat{u}}^{-1}A_{g_{\hat{u}}}) = (0,\dots,0)$, and it follows that $\operatorname{Ric}_{g_{\hat{u}}}\equiv 0$. 
	\end{proof}

	\subsection{Proof of the \textit{a priori} estimates in Theorem \ref{a}}\label{101}

	We now prove the \textit{a priori} estimate claimed in the first statement of Theorem \ref{a} for solutions to \eqref{2'}, which we restate here for convenience:

	\begin{prop}\label{E}
		Let $(M^n,g_0)$ be a smooth, closed Riemannian manifold of dimension $n\geq 3$ with $Y(M^n,[g_0])>0$. Let $h=h(x,z)$ be a smooth positive function satisfying (C2) and (C3), and suppose $(f,\Gamma)$ satisfies \eqref{A}, \eqref{B}, \eqref{C}, \eqref{D} and $\Gamma\not=\Gamma_1^+$. Suppose that there exists a metric $g_{\hat{u}} = e^{-2\hat{u}}g_0$, $\hat{u}\in C^0(M^n)$, satisfying \eqref{40}. Then for all $0<\delta<T<1$ and $\alpha\in (0,1)$, there exists a constant $C>0$ depending only on $n,g_0,\delta, T, \alpha$ and $h$ such that solutions $u_\tau$ to \eqref{2'} with $\tau\in[\delta,T]$ satisfy
		\begin{equation*}
		\| u_\tau \|_{C^{4,\alpha}(M^n,g_0)} \leq C. 
		\end{equation*}
	\end{prop} 
	
	\begin{rmk}
		By Proposition \ref{F}, the assumption $Y(M^n,[g_0])>0$ in Proposition \ref{E} is equivalent to $g_{\hat{u}}$ not being a smooth Ricci-flat metric. 
	\end{rmk}
	
	\begin{rmk}\label{r1}
		If we replace the RHS of \eqref{2'} with a $\tau$-dependent function $h_\tau(x,z)$, then the conclusion of Proposition \ref{E} still holds if (C2) is replaced by the conditions
		\begin{equation*}
		\limsup_{z\rightarrow-\infty}\sup_{x\in M,\,\tau\in[0,1]} h_\tau(x,z) = 0\quad\text{and}\quad \limsup_{z\rightarrow-\infty}\sup_{x\in M,\,\tau\in[0,1]}(|\nabla h_\tau(x,z)| + |\nabla^2 h(x,z)|)<\infty,
		\end{equation*} 
		and if (C3) is replaced by the condition $\liminf_{z\rightarrow+\infty} \inf_{x\in M,\,\tau\in[0,1]} h(x,z)= +\infty$.
	\end{rmk}

	We point out that the main task in Proposition \ref{E} is the $C^0$ estimate. First and second derivative estimates depending on $C^0$ estimates were established in works such as \cite{Via02, GW03b, LL03, GV03, Che05, Wan06, JLL07, Li09}, and once $C^2$ estimates are established, the equation \eqref{2'} becomes uniformly elliptic and higher order estimates follow from Evans--Krylov's theorem \cite{Ev82, Kry82} (here we use the concavity assumption in \eqref{C}) and Schauder estimates.

	\begin{proof}[Proof of Proposition \ref{E}]
		In light of the existing higher order estimates on solutions to \eqref{2'} discussed above, it suffices to prove $C^0$ bounds on solutions to \eqref{2'}. We split the proof into two steps: in the first step, we obtain the uniform upper bound in the range $\tau\in[\delta,1]$, and subsequently we apply the estimates of Chen \cite{Che05} to obtain uniform first and second derivative estimates in the range $\tau\in[\delta,1]$. In the second step we obtain the uniform lower bound in the range $\tau\in[\delta,T]$ for $T<1$. \newline 
		
		\noindent\textbf{Step 1:} We begin by proving the uniform upper bound on solutions to \eqref{2'}, which we will see holds uniformly in $\tau\in[\delta,1]$. 
		
		We first observe that by concavity, symmetry and homogeneity of $f$, for all $\lambda\in\Gamma^\tau$
		\begin{align*}
		f^\tau(\lambda) = f(\lambda^\tau) \leq f\bigg(\frac{\sigma_1(\lambda^\tau)}{n}e\bigg) + \nabla f\bigg(\frac{\sigma_1(\lambda^\tau)}{n}e\bigg)\cdot\bigg(\lambda^\tau - \frac{\sigma_1(\lambda^\tau)}{n}e\bigg) = \frac{f(e)}{n}\sigma_1(\lambda^\tau), 
		\end{align*}
		where we have used that $\nabla f(\frac{\sigma_1(\lambda^\tau)}{n}e)$ is parallel to $e$. Therefore, for any solution $u_\tau$ to \eqref{2'} we have (using the notation in \eqref{42})
		\begin{align}\label{28'}
		h(x,u_\tau(x)) \leq C\sigma_1(\lambda(g_0^{-1}A_{g_{u_\tau},\tau}))(x) \quad\text{for all }x\in M^n.
		\end{align}
		
		Now let $p\in M^n$ be a maximum point for $u_\tau$. At $p$, the gradient terms in \eqref{t87} vanish, and so \eqref{28'} implies
		\begin{align}\label{28}
		h(p,u_\tau(p))  & \leq   C\sigma_1(\lambda(g_0^{-1}A_{g_0,\tau}))(p) + C(\tau + n(1-\tau))\Delta_{g_0}u_\tau(p) \nonumber \\
		& \leq C\sigma_1(\lambda(g_0^{-1}A_{g_0,\tau}))(p),
		\end{align}
		where to obtain the last inequality we have used the fact $\Delta_{g_0}u_\tau(p)\leq 0$ (since $p$ is a maximum point). The growth condition (C3) then implies an upper bound for $u_\tau(p)$.
		
		With the uniform upper bound established, we may apply the first and second derivative estimates of Chen \cite{Che05}. Indeed, taking $W=A_{g_{u_\tau}}$, $F(g_0^{-1}W) = f^\tau(g_0^{-1}W)$ and $\Gamma = \Gamma^\tau$ in Case (a) of Theorem 1.1 therein, one obtains the first and second derivative estimate
		\begin{equation}\label{20}
		|\nabla_{g_0}^2 u_\tau|_{g_0} + |\nabla_{g_0} u_\tau|_{g_0}^2 \leq C \quad\text{on }M^n
		\end{equation}
		for all solutions $u_\tau$ to \eqref{2'}, $\tau\in[\delta,1]$, where $C$ depends on $n, g_0, \delta$ and an upper bound for
		\begin{align*}
		\sup_{x\in M,\, \tau\in[\delta,1]} \Big(&h\big(x,u_\tau(x)\big) + \big|\nabla h\big(x,u_\tau(x)\big)\big| + \big|\nabla^2 h\big(x,u_\tau(x)\big)\big|\Big),
		\end{align*}
		which by (C2) depends only on $h$ and the uniform upper bound for $u_\tau$ obtained above. \medskip
		
		\noindent\textbf{Step 2:} We now prove the uniform lower bound on solutions to \eqref{2'} for $\tau\in [\delta, T]$, $T<1$. 
		
		We assume for a contradiction that the uniform lower bound on solutions to \eqref{2'} fails, so that for some sequence $t_i\rightarrow t\in[\delta,T]$ we have a corresponding sequence of solutions $\{u_i\}$ such that $\min_{M^n} u_i\rightarrow-\infty$. By the uniform first derivative estimate from Step 1, it follows that $\max_{M^n} u_i\rightarrow-\infty$ as well. We denote $g_i = e^{-2u_i}g_0$ and define the rescaled sequence 
		\begin{equation*}
		\widetilde{u}_i = u_i - \bar{u}_i, \quad \bar{u}_i \defeq  \frac{1}{\operatorname{Vol}(M^n,g_0)}\int_{M^n} u_i \,dv_{g_0}.
		\end{equation*}
		Noting that $\widetilde{u}_i$ has zero average, the first and second derivative estimates in Step 1 imply that $\{\widetilde{u}_i\}$ is bounded in $C^2(M^n)$. Therefore, after restricting to a subsequence, for some $\widetilde{u}\in C^{1,1}(M^n)$ we have $\widetilde{u}_i\rightarrow \widetilde{u}$, where the convergence is in $C^{1,\alpha}(M^n)$ for all $\alpha<1$. Denoting $\widetilde{g}_i = e^{-2\widetilde{u}_i}g_0 = e^{2\bar{u}_i}g_i$, we observe that by homogeneity of $f$, the functions $\widetilde{u}_i$ satisfy
		\begin{align}\label{t15}
		f^{t_i}\big(\lambda(\widetilde{g}_i^{-1}A_{\widetilde{g}_i})\big)  = h(x,u_i) e^{2(u_i -\bar{u}_i)} \quad\text{on }M^n. 
		\end{align}

		Next observe that by our growth condition (C2), the RHS of \eqref{t15} tends to zero uniformly as $i\rightarrow \infty$. It follows that the metric $\widetilde{g} = e^{-2\widetilde{u}}g_0$ satisfies
		\begin{equation}\label{tt21}
		\lambda(\widetilde{g}^{-1}A_{\widetilde{g}}) \in \partial\Gamma^t \quad \text{in the viscosity sense on }M^n
		\end{equation}
		(see Proposition \ref{520} below). Since $\widetilde{u}\in C^{1,1}(M^n)$, it then follows (see e.g.~\cite[Lemma 2.5]{LNW20a}) from \eqref{tt21} that
		\begin{equation}\label{t21}
		\lambda(\widetilde{g}^{-1}A_{\widetilde{g}}) \in \partial\Gamma^t \quad \mathrm{a.e.~on~}M^n. 
		\end{equation}

		Now, by the assumption \eqref{40}, there exists $\hat{u}\in C^0(M^n)$, $g_{\hat{u}}= e^{-2\hat{u}}g_0$, satisfying
		\begin{equation}\label{t22}
		\lambda(g_{\hat{u}}^{-1}A_{g_{\hat{u}}}) \in \overline{\Gamma} \quad \text{in the viscosity sense on }M^n.
		\end{equation}
		It follows immediately from \eqref{t22} that for $t$ as above,
		\begin{equation*}
		\lambda(g_{\hat{u}}^{-1}A_{g_{\hat{u}}}) \in \overline{\Gamma^t} \quad \text{in the viscosity sense on }M^n.
		\end{equation*}

		We wish to show that $\hat{u}-\widetilde{u}$ is constant. To this end, let $c\in\mathbb{R}$ be such that $\hat{u} \leq \widetilde{u} + c$ on $M^n$ and $\hat{u}(x) = \widetilde{u}(x) + c$ for some $x\in M^n$. For this constant $c$, the set
		\begin{equation*}
		\mathcal{C} =\{x\in M^n: \hat{u}(x) = \widetilde{u}(x) + c\}
		\end{equation*}
		is therefore non-empty. By continuity of $\hat{u}$ and $\widetilde{u}$, $\mathcal{C}$ is also closed. Moreover, since $t<1$, for any $x\in \mathcal{C}$ we may apply the strong comparison principle of Theorem \ref{t20} (with $u_1 = \hat{u}$ and $u_2 = \widetilde{u}+c$) to a sufficiently small ball centred at $x$ and conclude that $\mathcal{C}$ is open. Therefore $\mathcal{C}= M^n$, i.e.~$\hat{u}  = \widetilde{u} + c$ on $M^n$. In particular, $\hat{u}\in C^{1,1}(M^n)$, and so by \cite[Lemma 2.5]{LNW20a} and \eqref{t22}, 
		\begin{equation}\label{t22'}
		\lambda(g_{\hat{u}}^{-1}A_{g_{\hat{u}}}) \in \overline{\Gamma} \quad \text{a.e. on }M^n. 
		\end{equation}
		
		Substituting $\widetilde{u} = \hat{u}-c$ into \eqref{t21}, we also see that $\lambda(g_{\hat{u}}^{-1}A_{g_{\hat{u}}})\in\partial\Gamma^t$ a.e.~on $M^n$. By \eqref{t22'} and Lemma \ref{506}, this implies that $\operatorname{Ric}_{g_{\hat{u}}}\equiv 0$ a.e.~on $M^n$, and taking the trace of this equation yields
		\begin{equation*}
		0 = R_{g_{\hat{u}}}e^{-2\hat{u}} =  R_{g_0} + 2(n-1)\Delta_{g_0} \hat{u} - (n-2)(n-1)|\nabla_{g_0} \hat{u}|_{g_0}^2 \quad\text{a.e. on }M^n.
		\end{equation*}
		Standard elliptic regularity (see e.g.~\cite[Theorem 9.19]{GT}) then implies $\hat{u}\in C^\infty(M^n)$, and thus $g_{\hat{u}}$ is a smooth metric with $R_{g_{\hat{u}}}\equiv 0$ on $M^n$. This contradicts positivity of $Y(M^n,[g_0])$, and the uniform lower bound is therefore established. 
	\end{proof}

	\begin{prop}\label{520}
		Suppose that $(f,\Gamma)$ satisfies \eqref{A}, \eqref{B}, \eqref{C} and \eqref{D}. Suppose $t_i\rightarrow t\in(0,1]$ and let $h_i\in C^0(M^n)$ be a sequence of positive functions converging uniformly to zero on $M^n$. Suppose that $g_{u_i} = e^{-2u_i}g_0$, $u_i\in C^0(M^n)$,  is a sequence of solutions to
		\begin{equation}\label{521}
		f^{t_i}(\lambda(g_{u_i}^{-1}A_{g_{u_i}})) = h_i, \quad \lambda(g_{u_i}^{-1}A_{g_{u_i}})\in\Gamma^{t_i} \quad \text{in the viscosity sense on }M^n,
		\end{equation}
		and that $u_i\rightarrow u$ uniformly. Then $g=e^{-2u}g_0$ satisfies 
		\begin{equation*}
		\lambda(g^{-1}A_g)\in\partial\Gamma^t \quad\text{in the viscosity sense on }M^n. 
		\end{equation*}
	\end{prop}

	\begin{proof}
		When $h_i=1$ for all $i$, the analogous result follows from the proof of \cite[Theorem 1.3]{LN20b}. A simple modification of this argument yields Proposition \ref{520}; we omit the details here. 
	\end{proof}

	\subsection{Proof of Theorem \ref{a}}\label{102}

	In this section we prove Theorem \ref{a}, from which Theorem \ref{f'} follows immediately.

	\begin{proof}[Proof of Theorem \ref{a}]
		We split the proof into three steps. In Step 1, under the assumption that $h$ satisfies (C1), we show that \eqref{2'} has at most one solution. In Step 2, we use a degree argument to prove the first statement in Theorem \ref{a}. In Step 3, we prove the second statement in Theorem \ref{a}. \medskip
		
		\noindent \textbf{Step 1:} Suppose that for fixed $\tau\in (0,1]$, $u$ and $v$ are two solutions to \eqref{2'}, and let $c\in\mathbb{R}$ be a constant such that $u\leq v + c$ on $M^n$ and $u(x) =  v(x) + c$ at some point $x\in M^n$. Since $f^\tau(\lambda(g_0^{-1}A_{g_{u}})) = f^\tau(\lambda(g_0^{-1}A_{g_{v+c}}))$, the strong comparison principle implies that $u=v+c$ on $M^n$. But this is compatible with (C1) if and only if $c=0$, and hence $u=v$. \medskip
		
		\noindent \textbf{Step 2:} In this step we prove the first statement in Theorem \ref{a}. By Proposition \ref{F}, either $g_{\hat{u}}$ is smooth and $\operatorname{Ric}_{g_{\hat{u}}} \equiv 0$, or $Y(M^n,[g_0])>0$. In the former case we are done, so suppose we are in the latter case. Without loss of generality, we assume that $R_{g_0}>0$.

		In contrast to the work of \cite{CD10, GV03, She08}, since our function is not assumed to satisfy the properness condition (C1), the continuity method is not applicable. We instead use a degree theory argument. Fix an arbitrary number $\alpha\in (0,1)$. Using the fact that $R_{g_0}>0$, we first fix $\delta>0$ for which $\lambda^\delta(g_0^{-1}A_{g_0})\in\Gamma_n^+$ and set $h_0= f^\delta(g_0^{-1}A_{g_0})>0$. Fix $T\in[\delta,1)$ and consider for $\tau\in[\delta,T]$, $g_{u_\tau}=e^{-2u_\tau}g_0,$ the equations
		\begin{equation}\label{ee1'}
		f^\tau\big(\lambda(g_0^{-1}A_{g_{u_\tau}})\big) = \frac{T-\tau}{T-\delta}h_0e^{2 u_\tau} + \frac{\tau-\delta}{T-\delta}h, \quad \lambda(g_0^{-1}A_{g_{u_\tau}}) \in \Gamma^\tau. 
		\end{equation}

		By Proposition \ref{E} and Remark \ref{r1}, there exists a positive constant $C$ such that every solution $u_\tau$ to \eqref{ee1'} with $\tau\in[\delta,T]$ satisfies $\|u_\tau\|_{C^{4,\alpha}(M^n)}\leq C/2$. For this constant $C$ and each $\tau\in[\delta,T]$, we then define
		\begin{equation}
		\mathcal{O}_\tau = \{u\in C^{4,\alpha}(M^n): \lambda(g_0^{-1}A_{g_u})\in\Gamma^\tau, \|u\|_{C^{4,\alpha}(M^n)}< C\}.
		\end{equation}
		Now denote the RHS of the equation in \eqref{ee1'} by $h^\tau$ and define
		\begin{equation*}
		F_\tau[x,u,\nabla u,\nabla^2 u] \defeq f^\tau(\lambda(g_0^{-1}A_{g_u}))(x) - h^\tau(x,u(x)),
		\end{equation*}
		so that solutions to \eqref{ee1'} are precisely the zeros of $F_\tau$. Then the degree $\operatorname{deg}(F_\tau,\mathcal{O}_\tau,0)$ in the sense of \cite{Li89} is well-defined an independent of $\tau\in[\delta,T]$. 
		
		We claim that $\operatorname{deg}(F_T,\mathcal{O}_T,0)=1$. By homotopy invariance, it suffices to show that $\operatorname{deg}(F_\delta,\mathcal{O}_\delta,0)=1$. To this end, first note that when $\tau=\delta$, $u_\delta\equiv0$ is the unique solution to \eqref{ee1'} with $\lambda(g_0^{-1}A_{g_{u_\delta}})\in\Gamma^\delta$, where uniqueness follows from Step 1. Therefore, by Propositions 2.3 and 2.4 in \cite{Li89}, to prove $\operatorname{deg}(F_\delta,\mathcal{O}_\delta,0)=1$, it suffices to show that the linearisation of $F_\delta$, as a mapping from $C^{2,\alpha}(M^n)$ to $C^\alpha(M^n)$, is invertible with no nonnegative eigenvalues. Indeed, for $u^s = u+s\phi$ we compute using \eqref{301}
		\begin{align}\label{43}
		\mathcal{L}^\delta(\phi) & \defeq \frac{d}{ds}\bigg|_{s=0} F_\delta[x,u^s, \nabla u^s, \nabla^2 u^s] = a_i^j (g_0^{-1}\nabla_{g_0}^2 \phi)^i_j + b^i(\nabla_{g_0} \phi)_i  + c\phi,
		\end{align}
		where $a^j_i = \delta L(g_0^{-1}A_{g_u,\delta})_i^j + (1-\delta)\sigma_1(L(g_0^{-1}A_{g_u,\delta}))\delta_i^j$ is positive definite by the ellipticity assumption in \eqref{D} (here $L$ denotes the linearisation of $f$), and $c = -\partial_u (h_0e^{2u})$ is negative. It follows that $\mathcal{L}^\delta$ is invertible as a mapping $\mathcal{L}^\delta:C^{2,\alpha}(M^n) \rightarrow C^{\alpha}(M^n)$ with no nonnegative eigenvalues, as required.

		We have shown that \eqref{ee1'} admits a solution for $\tau=T$, and since $T<1$ was arbitrary this completes the proof of the first statement in Theorem \ref{a}. \medskip

		\noindent \textbf{Step 3:} We now prove the second statement in Theorem \ref{a}. By Step 2, either there exists a smooth solution $g_{u_\tau}=e^{-2u_\tau}g_0$ to \eqref{2'} for each $\tau\in(0,1)$, or $g_{\hat{u}}$ is smooth with $\operatorname{Ric}_{g_{\hat{u}}} \equiv 0$. In the latter case, $g_{\hat{u}}$ clearly satisfies $\lambda(g_{\hat{u}}^{-1}A_{g_{\hat{u}}})\in\partial\Gamma$, and we are done.
		
		We now consider the former case. If there exists a $C^{1,1}$ metric $\hat{g}= e^{-2\hat{u}}g_0$ satisfying $\lambda({\hat{g}}^{-1}A_{\hat{g}})\in\partial\Gamma$ a.e.~on $M^n$, then the strong comparison principle implies that there is no smooth metric $g\in[g_0]$ satisfying $\lambda(g^{-1}A_g)\in\Gamma$ on $M^n$, and we are done. So suppose that there is no $C^{1,1}$ metric $\hat{g}= e^{-2\hat{u}}g_0$ satisfying $\lambda({\hat{g}}^{-1}A_{\hat{g}})\in\partial\Gamma$ a.e.~on $M^n$. As in Step 1 of the proof of Proposition \ref{E}, one obtains a uniform upper bound on solutions to \eqref{ee1'} uniformly in $\tau\in[\delta,1]$, and the first and second derivative estimates then follow uniformly in $\tau\in[\delta,1]$ by \cite{Che05}. Following the argument in Step 2 of Proposition \ref{E}, we then see that $\min_{M^n} u_\tau$ is bounded as $\tau\rightarrow 1$, otherwise one obtains a $C^{1,1}$ metric $\hat{g}= e^{-2\hat{u}}g_0$ satisfying $\lambda({\hat{g}}^{-1}A_{\hat{g}})\in\partial\Gamma$ a.e.~on $M^n$, a contradiction. The degree argument, as carried out in Step 2 above, therefore yields a smooth solution to \eqref{2'} with $\tau=1$. 
	\end{proof}

	\subsection{Proof of Theorem \ref{e}}\label{103}
	
	In this section we prove Theorem \ref{e}. Our solutions will be constructed as a suitably rescaled limit of solutions obtained in Theorem \ref{a}, therein taking $h(x,u) = e^{\beta u}$ and considering the limit $\beta\rightarrow 0^+$.

	\begin{proof}[Proof of Theorem \ref{e}]
		We only prove the first statement of Theorem \ref{e}, making use of the first statement in Theorem \ref{a}. The second statement in Theorem \ref{e} can be obtained in a similar way, by instead appealing to the second statement of Theorem \ref{a}. 
		
		By Theorem \ref{a}, either for each $\beta>0$ and $\tau\in(0,1)$ there exists a smooth solution $g_{u_{\tau,\beta}} = e^{-2u_{\tau,\beta}}g_0$ to \eqref{2'} with $h(x,z) = e^{\beta z}$, or $\hat{u}$ is smooth with $\operatorname{Ric}_{g_{\hat{u}}}\equiv 0$ on $M^n$. In the latter case we are done, so suppose we are in the former case. We fix $\tau\in(0,1)$ and henceforth write $u_\beta$ as shorthand for $u_{\tau,\beta}$. Letting $v_\beta = u_\beta - \bar{u}_\beta$, where $\bar{u}_\beta = \operatorname{Vol}(M^n,g_0)^{-1}\int_{M^n}u_\beta\,dv_{g_0}$, we see that
		\begin{equation}\label{31'}
		f^\tau\big(\lambda(g_0^{-1}A_{v_\beta})\big) = f^\tau\big(\lambda(g_0^{-1}A_{u_\beta})\big)  = e^{\beta u_\beta}.
		\end{equation}
		
		Let $p\in M^n$ be a maximum point for $u_\beta$ (equivalently, for $v_\beta$). As computed in Step 1 in the proof of Proposition \ref{E}, \eqref{31'} implies $e^{\beta u_\beta(p)} \leq C\sigma_1(\lambda(g_0^{-1}A_{g_0,\tau}))(p)$ and therefore
		\begin{equation}\label{32}
		e^{\beta u_\beta} \leq C \quad\text{on }M^n,
		\end{equation}
		where here and for the rest of the proof $C$ is a constant independent of $\beta$. As also discussed in the proof of Proposition \ref{E}, the \textit{a priori} first and second derivative estimates of \cite{Che05} on solutions $v_\beta$ to \eqref{31'} depend only on an upper bound for $e^{\beta u_\beta}$, and hence by \eqref{32} we have $|\nabla_{g_0} v_\beta|_{g_0}^2 + |\nabla_{g_0}^2 v_\beta|_{g_0} \leq C$. Since $v_\beta$ has zero average, we therefore have the full $C^2$ estimate
		\begin{equation}\label{33}
		\| v_\beta \|_{C^2(M^n,g_0)} \leq C. 
		\end{equation}
		
		Now, from \eqref{31'} we see that
		\begin{align}\label{31}
		f^\tau\big(\lambda(g_{v_\beta}^{-1}A_{v_\beta})\big) = e^{2v_\beta} 	f^\tau\big(\lambda(g_0^{-1}A_{v_\beta})\big) = e^{\beta\bar{u}_\beta} e^{(\beta+2) v_\beta }.
		\end{align}
		Moreover, since $\beta>0$, we can use Jensen's inequality to obtain from \eqref{32} the estimate
		\begin{equation*}
		e^{\beta\bar{u}_\beta} \leq C. 
		\end{equation*}
		If $\beta\bar{u}_\beta\rightarrow -\infty$ along some sequence $\beta\rightarrow 0$, then by \eqref{33}, \eqref{31} and Proposition \ref{520}, we get $C^{1,\alpha}$ convergence along a further subsequence of $v_\beta$ to some function $v_*\in C^{1,1}(M^n)$ satisfying
		\begin{equation*}
		\lambda(g_{v_*}^{-1}A_{g_{v_*}}) \in \partial\Gamma^\tau\quad  \text{ a.e. on }M^n,
		\end{equation*}
		which yields a contradiction exactly as in the proof of Proposition \ref{E}. Therefore, $\beta \bar{u}_\beta$ converges to some constant $c$ along a sequence $\beta\rightarrow 0$.  Again using \eqref{33} and \eqref{31}, we get $C^{1,\alpha}$ convergence of $v_\beta$ to some $v^*\in C^{1,1}(M^n)$, with $v^*$ satisfying
		\begin{equation}\label{34}
		f^\tau\big(\lambda(g_{v^*}^{-1}A_{g_{v^*}})\big) = e^{c} e^{2 v^*}, \quad \lambda(g_{v^*}^{-1}A_{g_{v^*}}) \in \Gamma^\tau  \quad \text{a.e. on }M^n. 
		\end{equation}
		By uniform ellipticity, $v^*$ is smooth and \eqref{34} is satisfied everywhere on $M^n$, which completes the existence part of the proof of the first statement with $\mu_\tau = e^{c}$. 
		
		We now prove the uniqueness part of Theorem \ref{e}. Assume for a contradiction that $\mu\not=\check{\mu}$ -- without loss of generality, we may assume that $\mu<\check{\mu}$. After adding a constant to one of our solutions if necessary, we may also assume that $u \leq \check{u}$ and $u(x) = \check{u}(x)$ for some $x\in M^n$. Then $f^\tau\big(\lambda(g_u^{-1}A_{g_u})\big) < f^\tau\big(\lambda(g_{\check{u}}^{-1}A_{g_{\check{u}}})\big)$ and the strong comparison principle implies $u=\check{u}$. This implies $\mu = \check{\mu}$, which is a contradiction.
	\end{proof}

	\subsection{A Kazdan-Warner type result}\label{p4}
	
	As a by-product of our method for establishing Theorem \ref{a}, we now prove a Kazdan-Warner type result which will be used in the proof of Theorem \ref{4'} in Section \ref{g2}. It is a simple consequence of the strong comparison principle that if $g\in[g_0]$ satisfies \eqref{79} and $\Gamma$ satisfies \eqref{A} and \eqref{B}, then every metric $\hat{g}\in[g_0]$ must satisfy $\lambda(\hat{g}^{-1}A_{\hat{g}}^t)\in\Gamma$ somewhere on $M^n$. We prove the following:

	\begin{thm}\label{201}
		Let $(M^n,g_0)$ be a smooth, closed Riemannian manifold of dimension $n\geq 3$ with $Y(M^n,[g_0])>0$, and suppose that $\Gamma$ satisfies \eqref{A} and \eqref{B}. Then there exists a smooth metric $g\in[g_0]$ satisfying $\lambda(g^{-1}A_g)\in\Gamma$ on $M^n$ if and only if no $C^{1,1}$ metric $g_{\hat{u}} = e^{-2\hat{u}}g_0$ with a.e.~nonnegative scalar curvature on $M^n$ is a solution to $\lambda(g_{\hat{u}}^{-1}A_{g_{\hat{u}}})\in\mathbb{R}^n\backslash\Gamma$ on $M^n$ in the viscosity sense. 
	\end{thm}

	\begin{proof}
		Suppose that there exists a smooth metric $g\in[g_0]$ satisfying $\lambda(g^{-1}A_g)\in\Gamma$ on $M^n$. Then by the strong comparison principle, there is no LSC metric $g_{\hat{u}} = e^{-2\hat{u}}g_0$ satisfying $\lambda(g_{\hat{u}}^{-1}A_{g_{\hat{u}}})\in\mathbb{R}^n\backslash\Gamma$ in the viscosity sense on $M^n$. 
		
		For the converse, suppose that no $C^{1,1}$ metric $g_{\hat{u}} = e^{-2\hat{u}}g_0$ with a.e.~nonnegative scalar curvature on $M^n$ is a solution to $\lambda(g_{\hat{u}}^{-1}A_{g_{\hat{u}}})\in\mathbb{R}^n\backslash\Gamma$ on $M^n$ in the viscosity sense. Since $Y(M^n,[g_0])>0$, we may assume that $R_{g_0}>0$ and, as in the proof of Theorem \ref{a}, we fix $\delta>0$ for which $\lambda^\delta(g_0^{-1}A_{g_0})\in\Gamma_n^+$ and set $h_0 = f^\delta(g_0^{-1}A_{g_0})>0$. Instead of \eqref{ee1'}, we then consider the path of equations 
		\begin{equation}\label{202}
		f^\tau(\lambda(g_0^{-1}A_{g_{u_\tau}}))  = h_0 e^{2u_\tau}, \quad \lambda(g_0^{-1}A_{g_{u_\tau}})\in\Gamma^\tau
		\end{equation}
		for $\tau\in[\delta,1]$, where $g_{u_\tau} = e^{-2u_\tau}g_0$.

		For any $\alpha\in(0,1)$, let 
		\begin{equation*}
		\mathcal{U} = \{\tau\in[\delta,1]:\eqref{202}\text{ has a solution }u_\tau\in C^{2,\alpha}(M^n)\}. 
		\end{equation*}
		Since $u_\tau\equiv 0$ is a solution to \eqref{202} when $\delta=0$, $\mathcal{U}$ is non-empty. 
		
		We now show that $\mathcal{U}$ is closed. First note that \eqref{202} falls into the framework of Proposition \ref{E}. As explained in Step 1 in the proof of Proposition \ref{E}, solutions $u_\tau$ to \eqref{202} admit an upper bound uniformly in $\tau\in[\delta,1]$, and consequently first and second derivative estimates uniformly in $\tau\in[\delta,1]$ by \cite{Che05}. Suppose for a contradiction the uniform lower bound on solutions to \eqref{202} fails for some sequence $t_i\rightarrow t_0 \leq 1$. Then, as explained in Step 2 in the proof of Proposition \ref{E}, one obtains a $C^{1,1}$ metric $g_{\hat{u}} = e^{-2\hat{u}}g_0$ satisfying $\lambda(g_{\hat{u}}^{-1}A_{g_{\hat{u}}})\in\partial\Gamma^{t_0}\subseteq \mathbb{R}^n\backslash\Gamma$ in the viscosity sense on $M^n$, a contradiction. This establishes a uniform $C^2$ estimate on solutions to \eqref{202}, and hence by \cite{Ev82, Kry82}, a uniform $C^{2,\beta}$ estimate for any $\beta\in(0,1)$. It follows that $\mathcal{U}$ is closed.
		
		It remains to show that $\mathcal{U}$ is open. But this follows immediately from the computation in \eqref{43}, which implies that the linearised operator corresponding to \eqref{202} is invertible as a mapping from $C^{2,\alpha}(M^n)$ to $C^\alpha(M^n)$. 
		
		We have shown that $\mathcal{U}$ is non-empty, closed and open, and it follows that $\mathcal{U} = [\delta,1]$. 
	\end{proof}

	\section{Some improvements when $(1,0,\dots,0)\in\Gamma$}\label{601}

	In this section we consider the following question: if the metric $g_{\hat{u}}$ in \eqref{40} is not a solution to the degenerate equation $\lambda(g_{\hat{u}}^{-1}A_{g_{\hat{u}}})\in\partial\Gamma$ in the viscosity sense on $M^n$, then does there exist a smooth conformal metric $g\in[g_0]$ satisfying $\lambda(g^{-1}A_g)\in\Gamma$ on $M^n$? We provide a positive answer to this question when $(1,0,\dots,0)\in\Gamma$ (in this case, we can appeal to our strong comparison principle in Theorem \ref{t20}$'$). For completeness, we restate Theorem \ref{e''} here to include also the case that $\hat{u}$ is smooth with $\lambda(g_{\hat{u}}^{-1}A_{g_{\hat{u}}})\in\Gamma$ on $M^n$:
	
	\begin{customthm}{\ref{e''}$'$}\label{e'}
		\textit{In addition to the hypotheses of Theorem \ref{e}, suppose also that one of the following conditions hold:\medskip 
			\begin{enumerate}
				\item[1.] $\hat{u}$ is smooth with $\lambda(g_{\hat{u}}^{-1}A_{g_{\hat{u}}})\in\Gamma$ on $M^n$, or \medskip
				\item[2.] $(1,0,\dots,0)\in\Gamma$ and $g_{\hat{u}}$ is not a solution to $\lambda(g_{\hat{u}}^{-1}A_{g_{\hat{u}}})\in\partial\Gamma$ on $M^n$ in the viscosity sense.\medskip
			\end{enumerate} 
			Then for all $\tau\in(0,1]$ there exists a constant $\mu_\tau>0$ and $u_\tau\in C^\infty(M^n)$ satisfying \eqref{5}. Moreover, if $(u_\tau,\mu_\tau)$, $(\check{u}_\tau,\check{\mu}_\tau) \in C^\infty(M^n)\times (0,\infty)$ both satisfy \eqref{5} then $\check{\mu}_\tau = \mu_\tau$ and $\check{u}_\tau=u_\tau+c$ for some constant $c\in\mathbb{R}$. }
	\end{customthm}
	
	As a means for establishing Theorem \ref{e'}, we will first prove the following refinement of the second statement in Theorem \ref{a}: 
	
	\begin{thm}\label{a3}
		In addition to the hypotheses of Theorem \ref{a}, suppose also that Condition 1 or 2 in Theorem \ref{e'} holds. Then for all $\tau\in(0,1]$ there exists a smooth solution $g_{u_\tau}=e^{-2u_\tau}g_0$ to \eqref{2'}. Moreover, for all $0<\delta<1$ and $\alpha\in(0,1)$, there exists a constant $C>0$ depending only on $n,g_0,\delta,\alpha$ and $h$ such that solutions to \eqref{2'} with $\tau\in[\delta,1]$ satisfy
		\begin{equation}\label{14b}
		\| u_\tau \|_{C^{4,\alpha}(M^n,g_0)} \leq C_2. 
		\end{equation}
		If $h$ also satisfies (C1), then solutions to \eqref{2'} are unique. 
	\end{thm}
	
	As an immediate corollary of Theorem \ref{a3} and Proposition \ref{60} in Appendix \ref{appb}, we obtain the following:

	\begin{cor}\label{f''}
		Let $(M^n,g_0)$ be a smooth, closed Riemannian manifold of dimension $n\geq 3$ and suppose $2\leq k \leq n$. Suppose for some fixed $t<1$ there exists a metric $g_{\hat{u}} = e^{-2\hat{u}}g_0$, $\hat{u}\in C^0(M^n)$, satisfying $\lambda(g_{\hat{u}}^{-1}A_{g_{\hat{u}}}^t)\in\overline{\Gamma_k^+}$ in the viscosity sense on $M^n$, but such that $g_{\hat{u}}$ is not a solution to $\lambda(g_{\hat{u}}^{-1}A_{g_{\hat{u}}}^t)\in\partial\Gamma_k^+$ in the viscosity sense on $M^n$. Then there exists a smooth metric $g_t\in[g_0]$ satisfying $\lambda(g_t^{-1}A_{g_t}^t)\in\Gamma_k^+$ on $M^n$.
	\end{cor}
	
	At the end of this section, we will also prove the uniqueness and regularity result stated in Corollary \ref{505'} in the introduction.

	The first step in the proof of Theorem \ref{a3} is to extend the \textit{a priori} estimates of Proposition \ref{E} to $\tau=1$, under either Condition 1 or 2 in Theorem \ref{e'}. We state these estimates as a separate result: 
	
	\begin{prop}\label{E2}
		In addition to the hypotheses of Proposition \ref{E}, suppose also that Condition 1 or 2 in Theorem \ref{e'} holds. Then for all $0<\delta<1$ and $\alpha\in(0,1)$, there exists a constant $C>0$ depending only on $n,g_0,\delta,\alpha$ and $h$ such that solutions $u_\tau$ to \eqref{2'} with $\tau\in[\delta,1]$ satisfy
		\begin{equation}\label{14'}
		\| u_\tau \|_{C^{4,\alpha}(M^n,g_0)} \leq C_2. 
		\end{equation}
	\end{prop}

	\begin{proof}[Proof of Proposition \ref{E2}]
		We first recall from Step 1 in the proof of Proposition \ref{E} that solutions to \eqref{2'} admit an upper $C^0$ bound and upper first and second derivative bounds uniformly in $\tau\in[\delta,1]$, and that neither Condition 1 nor 2 in the statement of Theorem \ref{e'} is needed here. By Step 2 in the proof of Proposition \ref{E}, for any $T<1$, solutions to \eqref{2'} for $\tau\in[\delta,T]$ also admit a lower bound; under only the hypotheses of Proposition \ref{E}, this lower bound may depend on $T$. To prove Proposition \ref{E2}, we only need to prove that the lower bound on solutions to \eqref{2'} holds uniformly as $\tau\rightarrow 1$, assuming either Condition 1 or 2. As before, the $C^{4,\alpha}$ estimate then follows from the theory of Evans-Krylov \cite{Ev82, Kry82} and classical Schauder estimates. 
		
		Let us first assume Condition 1, i.e.~we assume that there exists a smooth metric $g_{\hat{u}} = e^{-2\hat{u}}g_0$ satisfying $\lambda(g_{\hat{u}}^{-1}A_{g_{\hat{u}}})\in\Gamma$ on $M^n$. We suppose for a contradiction that the uniform lower bound on solutions to \eqref{2'} fails for some sequence $\tau_i\rightarrow 1$. Then, following the reasoning in Step 2 in the proof of Proposition \ref{E}, we obtain a $C^{1,1}$ metric $\widetilde{g} = e^{-2\widetilde{u}}g_0$ satisfying 
		\begin{equation}\label{t21'}
		\lambda(\widetilde{g}^{-1}A_{\widetilde{g}})\in\partial\Gamma \quad\text{a.e. on }M^n.
		\end{equation}
		By the strong comparison principle, this contradicts the fact that $g_{\hat{u}} = e^{-2\hat{u}}g_0$ is a smooth metric satisfying $\lambda(g_{\hat{u}}^{-1}A_{g_{\hat{u}}})\in\Gamma$ on $M^n$. Thus, the desired lower bound is established.
		
		Now we assume Condition 2, i.e.~we assume that $(1,0,\dots,0)\in\Gamma$, $g_{\hat{u}}$ satisfies \eqref{40} (that is, $\lambda(g_{\hat{u}}^{-1}A_{g_{\hat{u}}})\in\overline{\Gamma}$ in the viscosity sense on $M^n$) but $g_{\hat{u}}$ is not a solution to $\lambda(g_{\hat{u}}^{-1}A_{g_{\hat{u}}})\in\partial\Gamma$ in the viscosity sense on $M^n$. Assuming again for a contradiction that the lower bound on solutions to \eqref{2'} fails for a sequence $\tau_i\rightarrow1$, we obtain as before a function $\widetilde{u}\in C^{1,1}(M^n)$ satisfying \eqref{t21'}. Let $c\in\mathbb{R}$ be a constant such that $\hat{u} \leq \widetilde{u} + c$ on $M^n$ and $\hat{u}(x) = \widetilde{u}(x) + c$ for some $x\in M^n$. As in Step 2 in the proof of Proposition \ref{E}, the strong comparison principle in Theorem \ref{t20}$'$ then yields $\hat{u} = \widetilde{u} + c$ on $M^n$.

		On the other hand, by Condition 2 there exist a point $p$ and a function $\phi \in C^2(M^n)$ that touches $\hat{u}$ from below at $p$ such that $\lambda(g_{\phi}^{-1}A_{g_\phi})(p)\in\Gamma$. It follows that $\phi-c$ touches $\hat{u} - c = \widetilde{u}$ from below at $p$, and since $\widetilde{u}$ satisfies $\lambda(\widetilde{g}^{-1}A_{\widetilde{g}})\in\partial\Gamma$ in the viscosity sense on $M^n$, we therefore have $\lambda(g_{\phi-c}^{-1}A_{g_{\phi-c}})(p)\in\mathbb{R}^n\backslash\Gamma$. But this is equivalent to $\lambda(g_{\phi}^{-1}A_{g_\phi})(p)\in\mathbb{R}^n\backslash\Gamma$, a contradiction. This establishes the desired lower bound. 
	\end{proof}

	We now complete the proofs of Theorem \ref{a3}, Theorem \ref{e'} and Corollary \ref{505'}:
	
	\begin{proof}[Proof of Theorem \ref{a3}]
		The proof is identical to that of Theorem \ref{a}, except one takes $T=1$ in Step 2 and applies Proposition \ref{E2} instead of Proposition \ref{E}. 
	\end{proof}
	
	\begin{proof}[Proof of Theorem \ref{e'}]
		The proof is identical to that of Theorem \ref{e}, except one takes $u_\beta$ to be the solutions obtained in Theorem \ref{a3} with $h(x,u) = e^{\beta u}$. 
	\end{proof}
	
	\begin{proof}[Proof of Corollary \ref{505'}]
		Suppose that there exists a continuous viscosity solution $g_u = e^{-2u}g_0$ to the equation $\lambda(g_u^{-1}A_{g_u})\in\partial\Gamma$ on $M^n$. As a consequence of the strong comparison principle, it follows that there is no smooth metric $g\in[g_0]$ satisfying $\lambda(g^{-1}A_g)\in\Gamma$ on $M^n$. By the second statement in Theorem \ref{a}, it follows that there exists a $C^{1,1}$ metric $g=e^{-2w}g_0$ satisfying $\lambda(g^{-1}A_g)\in\partial\Gamma$ a.e.~on $M^n$. 
		
		Let $c\in\mathbb{R}$ and $x\in M^n$ be such that $u\leq w + c$ on $M^n$ and $u(x) = w(x) + c$. Following the argument in Step 2 in the proof of Proposition \ref{E}, the strong comparison principle in Theorem \ref{t20}$'$ then implies that $u = w + c$. In particular, this establishes the $C^{1,1}$ regularity of $u$. Now, if $g_v = e^{-2v}g_0$ is another continuous metric satisfying $\lambda(g_v^{-1}A_{g_v})\in\partial\Gamma$ in the viscosity sense on $M^n$, then the same argument yields $v = w + d$ on $M^n$ for some constant $d$, and therefore $u = v + c - d$ on $M^n$. This establishes the uniqueness statement.
	\end{proof}

	\section{Geometric applications}\label{g2}

	\subsection{The case $\Gamma = \Gamma_2^+$ and Ricci pinching (Theorem \ref{d'})}\label{800}
	
	As discussed in the introduction, part of our motivation for establishing Theorem \ref{f'} stems from earlier work on the existence of conformal metrics satisfying $\lambda(g^{-1}A_{g}^t)\in\Gamma_2^+$ on $M^n$. In \cite{CGY02a}, Chang, Gursky \& Yang established the existence of a metric $g\in[g_0]$ with $\lambda(g^{-1}A_g)\in\Gamma_2^+$ on any closed 4-manifold satisfying
	\begin{equation}\label{21}
	Y(M^4,[g_0])>0\quad \text{and} \quad \int_{M^4}\sigma_2(\lambda(g_0^{-1}A_{g_0}))\,dv_{g_0}>0. 
	\end{equation} 
	Note that, in light of the Chern-Gauss-Bonnet formula in four dimensions, the integral in \eqref{21} is conformally invariant.
	
	Ge, Lin \& Wang \cite{GLW10} later established a similar result on closed 3-manifolds: if 
	\begin{equation}\label{117}
	R_{g_0}>0\quad \text{and}\quad \int_{M^3}\sigma_2(\lambda(g_0^{-1}A_{g_0}))\,dv_{g_0}>0,
	\end{equation}
	then there exists a conformal metric $g\in[g_0]$ satisfying $\lambda(g^{-1}A_g)\in\Gamma_2^+$ on $M^3$. In fact, the authors obtain (see Case 1 of Theorem 2 therein) such a conformal metric in any dimension, assuming positivity of the following nonlinear Yamabe-type invariant:
	\begin{equation*}
	Y_{2,1}([g_0]) \defeq  \begin{cases}
	\begin{aligned}
	&\sup_{g\in[g_0],\,R_g>0} \int_{M^3}\sigma_2(\lambda(g^{-1}A_g))\,dv_g\int_{M^3}\sigma_1(\lambda(g^{-1}A_g))\,dv_g & \quad  \text{if }n=3\,\, \\[2pt]
	&\int_{M^4}\sigma_2(\lambda(g^{-1}A_g))\,dv_g & \quad \text{if }n=4\,\, \\[2pt]
	&\inf_{g\in[g_0],\,R_g>0}\frac{\int_{M^n}\sigma_2(\lambda(g^{-1}A_g))\,dv_g}{(\int_{M^n}\sigma_1(\lambda(g^{-1}A_g))\,dv_g)^\frac{n-4}{n-2}} & \quad \text{ if }n\geq 5,
	\end{aligned}
	\end{cases}
	\end{equation*} 
	with the convention that $Y_{2,1}([g_0])=-\infty$ if $Y(M^n,[g_0])\leq 0$. When $Y_{2,1}([g_0])=0$, they also established (see Case 2 of Theorem 2 therein) the existence of a $C^{1,1}$ metric $g$ conformal to $g_0$ satisfying $\lambda(g^{-1}A_g)\in\partial{\Gamma_2^+}$ a.e.~on $M^n$. See also \cite{CD10} in the case that $R_{g_0}>0$ and $\int_{M^3}\sigma_2(g_0^{-1}A_{g_0})\,dv_{g_0}\geq 0$. 
	
	We now give the proof of Theorem \ref{d'}: 
	
	\begin{proof}[Proof of Theorem \ref{d'}]
		The assumption $Y_{2,1}([g_0]) = 0$ implies, in particular, that the Yamabe invariant $Y(M^n,[g_0])$ is positive. By \cite[Theorem 2]{GLW10}, the assumption $Y_{2,1}([g_0]) = 0$ also implies the existence of a $C^{1,1}$ metric $g = e^{-2u}g_0$ satisfying
		\begin{equation}\label{204}
		\lambda(g^{-1}A_g)\in\partial{\Gamma_2^+}\quad\text{a.e. on }M^n.
		\end{equation} 
		By Corollary \ref{504'}, it follows from \eqref{204} that there exists a smooth metric $g_t\in[g_0]$ satisfying $\lambda(g_t^{-1}A_{g_t}^t)\in\Gamma_2^+$ for each $t<1$. Moreover, the non-existence of a smooth metric $g\in[g_0]$ satisfying $\lambda(g^{-1}A_g)\in\Gamma_2^+$ on $M^n$ follows from \eqref{204} and the strong comparison principle. 
	\end{proof}

	The geometric significance of Theorem \ref{d'} is as follows. It is well-known in dimensions $n\geq 3$ that the Ricci tensor of a metric satisfying $\lambda(g^{-1}A_g^t)\in\Gamma_2^+$ on $M^n$ satisfies the following pointwise pinching property:
	\begin{equation}\label{1}
	\bigg(t - 2 + \frac{4}{n}\bigg)R_g g < 2\operatorname{Ric}_g < (2-t)R_g g
	\end{equation}
	(see e.g.~\cite[Proposition 5.2]{GV03} for a proof of this fact). In this way, Theorem \ref{d'} asserts the existence of conformal metrics satisfying \textit{pointwise} pinching of the Ricci tensor, given only an \textit{integral} pinching condition on a background metric of positive scalar curvature. We state this pinching result as a further corollary:
	
	\begin{cor}\label{23}
		Let $(M^n,g_0)$ be a smooth, closed Riemannian manifold of dimension $n\geq 3$ with $Y_{2,1}([g_0]) \geq 0$. Then for all $t<1$, there exists a smooth metric $g\in[g_0]$ satisfying $R_g>0$ and \eqref{1}.
	\end{cor}

	It is known that in dimension $n=3$, the pinching relation in \eqref{1} with $t\geq \frac{2}{3}$ implies that $M^3$ is diffeomorphic to a spherical space form.

	\subsection{Differential inclusions under positivity of Yamabe-type invariants}\label{52}
	
	In this section we give an application of Theorem \ref{201}, in which we prove the existence of a smooth conformal metric satisfying $\lambda(g^{-1}A_g^t)\in\Gamma_2^+$ on $M^n$ assuming positivity of a nonlinear Yamabe-type invariant. We restrict our attention to the case $n\geq 5$; for the cases $n=3$ and $n=4$ we refer to \cite{CGY02a, GV03, GLW10, CD10}. 
	
	More precisely, for $t\leq 1$ and $n\geq 5$ we define 
		\begin{equation}\label{30'}
	Y_2^t([g_0]) =  \inf_{g\in[g_0],\, R_g>0}  \frac{\int_{M^n}\sigma_2(\lambda(g^{-1}A_g^t))\,dv_g}{\operatorname{Vol}(M^n,g)^{\frac{n-4}{n}}},
	\end{equation}
	with the convention that $Y_{2}^t([g_0]) = -\infty$ if $Y(M^n,[g_0])\leq 0$. We prove:

	\begin{thm}\label{4'}
		Let $(M^n,g_0)$ be a smooth, closed Riemannian manifold of dimension $n\geq 5$, and suppose $t\leq 1$. If $Y_{2}^t([g_0]) >0$, then there exists a smooth metric $g\in[g_0]$ satisfying $\lambda(g^{-1}A_g^t)\in\Gamma_2^+$ on $M^n$.  
	\end{thm}
	
	\begin{rmk}
		In the case $t=1$, Theorem \ref{4'} was previously established by Sheng in \cite{She08}. In this case, Theorem \ref{4'} implies the result of Ge, Lin \& Wang \cite{GLW10} (discussed in the previous subsection) in dimensions $n\geq 5$, since $Y_{2,1}([g_0])>0$ implies $Y_2^1([g_0]) \geq CY_{2,1}([g_0])Y(M^n,[g_0])^{\frac{n-4}{n-2}}>0$. In the case $t<1$, a related statement was also proved in \cite{She08}, although a condition different to $Y_2^t([g_0])>0$ was considered therein.
	\end{rmk}

	\begin{proof}[Proof of Theorem \ref{4'}]
		By Theorem \ref{201} and the fact that $C^{1,1}$ functions are a.e.~punctually second order differentiable, it suffices to show that if $Y_{2}^t([g_0]) >0$, then the RHS of \eqref{30'} is still positive if the infimum is instead taken over the set of all metrics $g = {u}^{\frac{4}{n-2}}g_0$ with $0<{u}\in C^{1,1}(M^n)$ and $R_g\geq0$ a.e.~on $M^n$. By the definition of $Y_{2}^t([g_0])$, we may assume that $R_{g_0}>0$. 
		
		To this end, let $g$ be such a metric. If $L_{g_0}$ denotes the conformal Laplacian of $g_0$, then $u$ satisfies the equation
		\begin{equation*}
		L_{g_0} u  = f \defeq R_{g}u^{\frac{n+2}{n-2}}\quad \text{a.e. on }M^n. 
		\end{equation*}
		Observe that $f\geq 0$ is not identically zero, otherwise we would have $u\equiv 0$ in view of the fact $R_{g_0}>0$. Since $f\geq 0$ and $f$ belongs to $L^\infty(M^n)$, we may convolve $f$ with nonnegative smooth mollifiers to obtain a sequence $f_i$ of smooth, nonnegative functions on $M^n$ such that $f_i\rightarrow f$ in $L^p(M^n,g_0)$ for all $p<\infty$. After passing to a subsequence, we may assume that none of the $f_i$ are identically zero. Let $u_i$ denote the corresponding smooth solutions to 
		\begin{equation*}
		L_{g_0} u_i = f_i \quad \text{on }M^n. 
		\end{equation*}
		Note that, since $R_{g_0}>0$ and the $f_i\geq 0$ are not identically zero, the maximum principle implies $u_i>0$. Moreover, since $f_i\rightarrow f$ in $L^p(M^n,g_0)$ for all $p<\infty$, standard elliptic theory implies that $u_i\rightarrow u$ in $W^{2,p}(M^n,g_0)$ for all $p<\infty$, and hence
		\begin{equation*}
		Y_{2}^t([g_0]) \leq \frac{\int_{M^n}\sigma_2(\lambda(g_{u_i}^{-1}A^t_{g_{u_i}}))\,dv_{g_{u_i}}}{\operatorname{Vol}(M^n,g_{u_i})^\frac{n-4}{n}} \longrightarrow \frac{\int_{M^n}\sigma_2(\lambda(g_{u}^{-1}A^t_{g_{u}}))\,dv_{g_{u}}}{\operatorname{Vol}(M^n,g_u)^{\frac{n-4}{n}}}
		\end{equation*}
		as $i\rightarrow \infty$, as required.
	\end{proof}

	\subsection{Conformal invariance of the sign of $\lambda(\sigma_2,g_0)$ (Theorem \ref{61'})}\label{801}
	
	In this section we prove Theorem \ref{61'}. In \cite{GLW10}, the authors considered the following nonlinear eigenvalue for the $\sigma_2$ operator:
	
	\begin{equation*} 
	\lambda(\sigma_2,g_0) \defeq \begin{cases}
	\begin{aligned} &\sup_{g=e^{-2u}g_0,\, R_g>0}\frac{\int_{M^3}\sigma_2(\lambda(g^{-1}A_g))\,dv_g}{\int_{M^3}e^{4u}\,dv_g}   & \quad  \text{if }n=3 \\[5pt]
	&\int_{M^4}\sigma_2(\lambda(g^{-1}A_g))\,dv_g & \quad \text{if }n=4 \\[5pt]
	&\inf_{g=e^{-2u}g_0,\,R_g>0}\frac{\int_{M^n}\sigma_2(\lambda(g^{-1}A_g))\,dv_g}{\int_{M^n}e^{4u}\,dv_g}   & \quad  \text{if }n > 5.
	\end{aligned} 
	\end{cases} \bigskip
	\end{equation*}

	The interpretation of $\lambda(\sigma_2,g_0)$ as a nonlinear eigenvalue comes from \cite[Theorem 1]{GLW10}, where it is shown that if $\lambda(\sigma_2,g_0) >0$, then $\lambda(\sigma_2,g_0)$ is achieved by a smooth metric $g_u=e^{-2u}g_0$ of positive scalar curvature satisfying $\sigma_2(\lambda(g_u^{-1}A_{g_u})) = \lambda e^{4u}$ on $M^n$. It is also shown that when $\lambda(\sigma_2,g_0)=0$, $\lambda(\sigma_2,g_0)$ is achieved by a $C^{1,1}$ metric $g_u$ of nonnegative scalar curvature satisfying $\sigma_2(\lambda(g_u^{-1}A_{g_u})) = 0$ a.e.~on $M^n$. 
	
	In analogy with the case for the scalar curvature, where the sign of the Yamabe invariant of $[g_0]$ coincides with the sign of the first eigenvalue of the conformal Laplacian of any metric in $[g_0]$, one may expect a relationship between the signs of $Y_{2,1}([g_0])$ and $\lambda(\sigma_2,g_0)$. It is shown in \cite[Lemma 3]{GLW10} that the sign of $Y_{2,1}([g_0])$ coincides with the sign of $\lambda(\sigma_2,g_0)$ when $n\geq 4$. For $n=3$, it is also shown $\lambda(\sigma_2,g_0)>0$ if and only if $Y_{2,1}([g_0])>0$, $\lambda(\sigma_2,g_0) \leq 0$ if and only if $Y_{2,1}([g_0]) \leq 0$, and $\lambda(\sigma_2,g_0)<0$ implies $Y_{2,1}([g_0])<0$. Using our existence result in Theorem \ref{a}, we prove the remaining implication and hence establish Theorem \ref{61'}:
	
	\begin{proof}[Proof of Theorem \ref{61'}]
		In light of the discussion above, we only need to show that $\lambda(\sigma_2,g_0)=0$ implies $Y_{2,1}([g_0])=0$. For shorthand we denote
		\begin{equation*}
		F[g] = \int_{M^3}\sigma_2(\lambda(g^{-1}A_g))\,dv_g\int_{M^3}\sigma_1(\lambda(g^{-1}A_g))\,dv_g
		\end{equation*}
		so that $Y_{2,1}([g_0]) = \sup_{g\in[g_0],\,R_g>0}F[g]$. 
		
		For $\tau<1$, let $g_{u_\tau} = e^{-2u_\tau}g_0$ be the solutions obtained to \eqref{ee1'} in the proof of Theorem \ref{a} in the case that $f = \sigma_2^{1/2}$ and $Y(M^3,[g_0])>0$. A routine calculation yields
		\begin{equation*}
		\sigma_2(\lambda(g_{u_\tau}^{-1}A_{g_{u_\tau}}))  =  (1-\tau)c(n,\tau)[\sigma_1(\lambda(g_{u_\tau}^{-1}A_{g_{u_\tau}}))]^2 -  \tau^{-2}\sigma_2^\tau(\lambda(g_{u_\tau}^{-1}A_{g_{u_\tau}})),
		\end{equation*}
		where $c(n,\tau) = \frac{1}{2}\tau^{-2}(n-1)(2\tau+n(1-\tau))$ is bounded as $\tau\rightarrow 1$. Therefore
		\begin{align}
		F[g_{u_\tau}] & > -C(1-\tau) \int_{M^3}[\sigma_1(\lambda(g_{u_\tau}^{-1}A_{g_{u_\tau}}))]^2\,dv_{g_{u_\tau}}\int_{M^3}\sigma_1(\lambda(g_{u_\tau}^{-1}A_{g_{u_\tau}}))\,dv_{g_{u_\tau}}, \label{8}
		\end{align}
		where here and for the remainder of the proof, $C$ is a constant that remains bounded as $\tau\rightarrow 1$ (but may change from line to line). 
		
		Next, observe that we have the identities
		\begin{align*}
		R_{g_{u_\tau}}^2\,dv_{g_{u_\tau}} & = e^{u_\tau}\big(R_{g_0} + 4\Delta_{g_0} u_\tau - 2|\nabla_{g_0} u_\tau|_{g_0}^2\big)^2\,dv_{g_0}, \nonumber \\
		R_{g_{u_\tau}}\,dv_{g_{u_\tau}} & = e^{-u_\tau}(R_{g_0}+ 4\Delta_{g_0} u_\tau - 2|\nabla_{g_0} u_\tau|_{g_0}^2\big)\,dv_{g_0}.
		\end{align*}
		Moreover, by the first and second derivative estimates obtained on $u_\tau$ in Step 1 of the proof of Proposition \ref{E} (which, we recall, hold uniformly for $\tau\leq 1$), we have $0< R_{g_0} + 4\Delta_{g_0} u_\tau - 2|\nabla_{g_0}u_\tau|_{g_0}^2 \leq C$, and it then follows from the above that
		\begin{equation}\label{3}
		F[g_{u_\tau}] \geq -C(1-\tau) \int_{M^3}e^{u_\tau}\,dv_{g_0} \int_{M^3}e^{-u_\tau}\,dv_{g_0}.
		\end{equation}
		
		We claim that the product of integrals in \eqref{3} is bounded independently of $\tau\leq 1$. Once this is established, it will follow that $F[g_{u_\tau}] \geq -C(1-\tau)$. Since we already know by \cite[Lemma 3]{GLW10} that $Y_{2,1}([g_0]) = \sup_{g\in[g_0],\,R_g>0} F[g] \leq 0$, and since $\tau<1$ is arbitrary, the conclusion $Y_{2,1}([g_0])=0$ then follows.
		
		To prove the claim, observe by the uniform first derivative estimates obtained on $u_\tau$ in Step 1 of the proof of Proposition \ref{E} (which are uniform for $\tau\leq 1$), there exists a constant $C$ such that $|u_\tau(x) - u_\tau(y)| \leq C$ for all $x,y\in M^3$ and all $\tau<1$, i.e. 
		\begin{equation}\label{302}
		-C - u_\tau(y) \leq -u_\tau(x) \leq C - u_\tau(y) \quad \text{for all } x,y\in M^3. 
		\end{equation}
		Taking exponentials in \eqref{302} (which preserves the inequalities) and then integrating against $dv_{g_0}(y)$ gives 
		\begin{equation*}
		0 < C^{-1}\int_{M^3} e^{-u_\tau}\,dv_{g_0} \leq e^{-u_\tau(x)} \leq C\int_{M^3} e^{-u_\tau}\,dv_{g_0} \quad  \text{for all } x\in M^3, 
		\end{equation*}
		or equivalently 
		\begin{equation*}
		0 < C^{-1} \leq e^{u_\tau(x)} \int_{M^3} e^{-u_\tau}\,dv_{g_0} \leq C \quad  \text{for all }x\in M^3. 
		\end{equation*}
		Integrating against $dv_{g_0}(x)$ then gives
		\begin{equation*}
		0 < C^{-1} \leq \int_{M^3} e^{u_t}\,dv_{g_0} \int_{M^3}e^{-u_t}\,dv_{g_0} \leq C,
		\end{equation*}
		as claimed. This completes the proof. 
	\end{proof} 
	
	\subsection{A generalisation of  a theorem of Aubin \& Ehrlick (Theorem \ref{700})}\label{802}

	\begin{proof}[Proof of Theorem \ref{700}]
		The assumption \eqref{312} is equivalent to $\lambda(\hat{g}^{-1}A_{\hat{g}}^t)\in\overline{\Gamma_n^+}$ in the viscosity sense on $M^n$ for $t = \frac{\alpha}{2(n-1)}$. Since we assume that $\hat{g}$ is not a $C^{1,1}$ solution to $\lambda(\hat{g}^{-1}A_{\hat{g}}^t)\in\partial\Gamma_n^+$ a.e.~on $M^n$, the desired conclusion follows from Theorem \ref{e''} and Remark \ref{505''}, since $t<1$.
	\end{proof}

	As a by-product of the proof of Theorem \ref{700}, the metric $g$ obtained in Theorem \ref{700} satisfies $\operatorname{det}(g_0^{-1}(\operatorname{Ric}_g - \alpha R_g g)) = c$ on $M^n$ for some constant $c>0$. When $0 \leq \alpha < \frac{1}{2(n-1)}$, as an alternative to the above proof, one can use \cite{LN14} together with the strong comparison principle Theorem \ref{t20} to solve the trace-modified $\sigma_n$-Yamabe equation $\operatorname{det}(g^{-1}(\operatorname{Ric}_g - \alpha R_g g))=1$ on $M^n$ as follows: the proof of \cite[estimate (9)]{LN14} shows that, as $\alpha\geq 0$, all solutions to the trace-modified $\sigma_n$-Yamabe equation are bounded from below. If they are not bounded from above, then a suitably rescaled limit will yield a $C^{1,1}$ solution to the equation $\lambda(g^{-1}(\operatorname{Ric}_g - \alpha R_g g))\in\partial\Gamma_n^+$. By the strong comparison principle in Theorem \ref{t20}, this contradicts the assumptions on $\hat{g}$ in Theorem \ref{312}. Existence then follows from a degree theory argument.

	\subsection{Differential inclusions vs.~nonlinear Green's functions (Theorem \ref{305})}\label{803}

	\begin{proof}[Proof of Theorem \ref{305}]
		We assume for simplicity that the Green's function \linebreak $w\in C^0_{\operatorname{loc}}(M^n\backslash \{p\})$ has a single pole at $p\in M^n$, of fixed but arbitrary strength. Observe that, if $v$ is a positive function on $M^n$ satisfying $\nabla_{g_0}v(p) = 0$ and $\nabla_{g_0}^2 v(p)<0$, then by \eqref{301} with $u=-\frac{2}{n-2}\log v$, the Schouten tensor of the metric $\alpha v^{\frac{4}{n-2}}g_0$ is positive definite near $p$ for $\alpha$ large. In particular, after making a change of background metric if necessary, we may assume that $g_0$ satisfies $\lambda(g_0^{-1}A_{g_0})\in\Gamma$ on some small geodesic ball $B_\ep(p)$ centred at $p$. 
		
		Fix a large constant $M>1+\max_{M^n\backslash B_\ep(p)}w$ and define $\widetilde{w}=\min\{w,M\}$. Clearly $\widetilde{w}$ is continuous, and we claim that the metric $\widetilde{g} = \widetilde{w}^{\frac{4}{n-2}}g_0$ satisfies $\lambda(\widetilde{g}^{-1}A_{\widetilde{g}})\in\overline{\Gamma}$ in the viscosity sense on $M^n$. Once this claim is established, the existence of a smooth metric $g\in[g_0]$ satisfying $\lambda(g^{-1}A_g)\in\Gamma$ on $M^n$ follows immediately from Theorem \ref{a3}. 
		
		To this end, fix $x\in M^n$ and suppose that $\phi\in C^2(M^n)$ is a positive function touching $\widetilde{w}$ from below at $x$. To prove the claim, we need to show that the metric $\bar{g} = \phi^{\frac{4}{n-2}}g_0$ satisfies $\lambda(\bar{g}^{-1}A_{\bar{g}})(x)\in\overline{\Gamma}$. There are three cases to consider: \medskip 
		
		\textit{Case 1:} $x\in\{w<M\}$. Since $w$ is continuous away from $p$, the set $\{w<M\}$ is open and hence $\widetilde{w}\equiv w$ in a neighbourhood of $x$. Therefore, $\phi$ also touches $w$ from below at $x$, and it follows from the fact that $w$ is a viscosity subsolution that $\lambda(\bar{g}^{-1}A_{\bar{g}})(x)\in\overline{\Gamma}$. \medskip
		
		\textit{Case 2:} $x\in \{w>M\}$. Since $\widetilde{w}(x)= M$ and $\lambda(g_0^{-1}A_{g_0})\in\Gamma$ in $B_\ep(p)$, it follows in this case that $\lambda(\widetilde{g}^{-1}A_{\widetilde{g}})(x)\in\overline{\Gamma}$. Consequently, $\lambda(\bar{g}^{-1}A_{\bar{g}})(x)\in\overline{\Gamma}$. \medskip
		
		\textit{Case 3:} $x\in\{w=M\}$. We have $\phi(x) = \widetilde{w}(x) = M  = w(x)$ and $\phi < \widetilde{w} \leq w$ near $x$ (note $\widetilde{w}\leq w$ on all of $M^n$, by definition of $\widetilde{w}$). Therefore, $\phi$ touches $w$ from below at $x$, and the conclusion follows as in Case 1. 
	\end{proof}

	\begin{appendix}
		\section{A characterisation of cones satisfying $(1,0,\dots,0)\in\Gamma$}\label{appb}
		
		In this appendix we give an equivalent characterisation of cones $\Gamma$ satisfying \eqref{A}, \eqref{B} and $(1,0,\dots,0)\in\Gamma$:
		
		\begin{prop}\label{60}
			Suppose $\Gamma$ satisfies \eqref{A} and \eqref{B}. Then $(1,0,\dots,0)\in\Gamma$ if and only if there exists $\widetilde{\Gamma}$ satisfying \eqref{A} and \eqref{B} and a number $\tau<1$ for which $\Gamma = (\widetilde{\Gamma})^\tau$. 
		\end{prop}
		\begin{proof}
			Suppose first of all that $\Gamma  = (\widetilde{\Gamma})^\tau$ for some $\widetilde{\Gamma}$ satisfying \eqref{A} and \eqref{B} and some $\tau<1$. Then 
			\begin{equation}\label{600}
			\tau(1,0,\dots,0) + (1-\tau)\sigma_1(1,0,\dots,0) e = (1,1-\tau,\dots,1-\tau)\in\Gamma_n^+,
			\end{equation}
			and the assertion $(1,0,\dots,0)\in\Gamma$ then follows from \eqref{600} and the assumption $\Gamma_n^+ \subseteq \Gamma$. 
			
			Conversely, suppose $\Gamma$ satisfies \eqref{A}, \eqref{B} and $(1,0,\dots,0)\in\Gamma$. We first claim that for $\tau'>1$ sufficiently close to 1, $\Gamma^{\tau'}$ satisfies \eqref{A} and \eqref{B}. In fact, \eqref{A} is immediate, and the inclusion $\Gamma^{\tau'}\subseteq\Gamma_1^+$ in \eqref{B} follows from the fact $\Gamma^{\tau'}\subseteq \Gamma \subseteq \Gamma_1^+$. So it remains to show that $\Gamma_n^+\subseteq\Gamma^{\tau'}$ for $\tau'>1$ sufficiently close to 1. To this end, first observe that for $\tau'>1$ sufficiently close to 1, $(1,0,\dots,0)\in\Gamma^{\tau'}$. Indeed, $(1,0,\dots,0)\in\Gamma^{\tau'}$ if and only if $(1,1-\tau',\dots,1-\tau')\in\Gamma$, and this latter inclusion is clearly seen to hold for $\tau'>1$ sufficiently close to 1, using openness of $\Gamma$ and the fact that $(1,0,\dots,0)\in\Gamma$. Now, since $\Gamma^{\tau'}$ is convex, symmetric and satisfies $(1,0,\dots,0)\in \Gamma^{\tau'}$, the convex hull of all permutations of $(1,0,\dots,0)$ is also contained in $\Gamma^{\tau'}$. By homothety, it then follows that $\Gamma_n^+\subseteq\Gamma^{\tau'}$, as required.

			By the previous paragraph, we may fix $\tau'>1$ sufficiently close to 1 so that $\Gamma^{\tau'}$ satisfies \eqref{A} and \eqref{B}. Then define $\widetilde{\Gamma} = \Gamma^{\tau'}$, and observe that $\Gamma = (\widetilde{\Gamma})^{\widetilde{\tau}}$ for 
			\begin{equation}\label{37}
			\widetilde{\tau} = \frac{n-(n-1)\tau'}{(n-1) - (n-2)\tau'}<1.
			\end{equation}
			Indeed, for any $\tau$, $\lambda \in (\widetilde{\Gamma})^\tau$ if and only if $\tau\lambda + (1-\tau)\sigma_1(\lambda) e \in \widetilde{\Gamma}$ which occurs if and only if 
			\begin{equation}\label{36}
			\tau'\big[\tau\lambda + (1-\tau)\sigma_1(\lambda)e\big] + (1-\tau')\sigma_1\big[\tau\lambda + (1-\tau)\sigma_1(\lambda)e\big]e\in\Gamma, 
			\end{equation}
			since $\widetilde{\Gamma} = \Gamma^{\tau'}$. Collecting coefficients, we see that \eqref{36} is equivalent to 
			\begin{equation}\label{38}
			\tau'\tau \lambda + \big[\tau'(1-\tau) + \tau(1-\tau') + n(1-\tau)(1-\tau')\big]\sigma_1(\lambda) e \in \Gamma. 
			\end{equation}
			Taking $\tau = \widetilde{\tau}$ in \eqref{38}, we see that quantity in the square parentheses vanishes, and hence  $\lambda\in(\widetilde{\Gamma})^{\widetilde{\tau}}$ if and only if $\tau'\widetilde{\tau}\lambda\in\Gamma$, which occurs if and only if $\lambda \in \Gamma$ (since $a\Gamma = \Gamma$ for all $a>0$). 
		\end{proof}
	\end{appendix}

	\noindent\textbf{Acknowledgements:} Part of this work was carried out as part of JD's DPhil thesis, which was supported by EPSRC grant number EP/L015811/1. Both authors would like to thank Yuxin Ge, YanYan Li and Guofang Wang for insightful comments and useful suggestions.

	\bibliography{references}{}
	\bibliographystyle{siam}

\end{document}